\documentclass[11pt]{amsart}

\usepackage{amssymb,latexsym}

\usepackage{graphicx,epsfig}

\usepackage[dvipsnames]{xcolor}

\def\cal{\mathcal}

\def\Bbb{\mathbb}

\def\N{{\Bbb N}}
\def\Z{{\Bbb Z}}
\def\R{{\Bbb R}}
\def\C{{\Bbb C}}	

\def\W{{\cal W}}
\def\Y{{\cal Y}}
\def\V{{\cal V}}

\def\F{{\cal F}}	
\def\ext{{\cal E}} 
\def\cP{{\cal P}}
\def\D{{\cal D}}

\def\cP{{\cal P}}
\def\I{{\cal I}}
\def\J{{\cal J}}
\def\cR{{\cal R}}

\def\Landau{{\cal O}}

 \def\Om{{\Omega}}

\def\de{{\delta}}
\def\la{{\lambda}}
\def\vp{{\varphi}}
\def\eps{{\epsilon}}

\def \dist {\text{\rm dist\,}}
\def \supp {\text{\rm supp\,}}
\def \graph {\text{\rm graph\,}}

\def\e{\varepsilon}
\def\ve{\varepsilon}
\def\trans{{\,}^t}

\def\pa{{\partial}}

\def\cal{\mathcal}

\def\Bbb{\mathbb}

\def\labelenumi{(\roman{enumi})}

\newtheorem{thm}{Theorem}[section]

\newtheorem{cor}[thm]{Corollary}
\newtheorem{lemma}[thm]{Lemma}
\newtheorem{remark}[thm]{Remark}
\newtheorem{remarks}[thm]{Remarks}
\newtheorem{defn}[thm]{Definition}

\newtheorem{example}[thm]{Example}

\begin{document}

\title[ flat  perturbations of the hyperbolic paraboloid]{Partitions of flat one -variate functions and a Fourier restriction theorem for related  perturbations of the hyperbolic paraboloid}

\author[S. Buschenhenke]{Stefan Buschenhenke}
\address{S. Buschenhenke:  Mathematisches Seminar, C.A.-Universit\"at Kiel,
Ludewig-Meyn-Stra\ss{}e 4, D-24118 Kiel, Germany}
\email{{\tt buschenhenke@math.uni-kiel.de}}
\urladdr{http://www.math.uni-kiel.de/analysis/de/buschenhenke}

\author[D. M\"uller]{Detlef M\"uller}
\address{D. M\"uller: Mathematisches Seminar, C.A.-Universit\"at Kiel,
Ludewig-Meyn-Stra\ss{}e 4, D-24118 Kiel, Germany}
\email{{\tt mueller@math.uni-kiel.de}}
\urladdr{http://www.math.uni-kiel.de/analysis/de/mueller}

\author[A. Vargas]{Ana Vargas}
\address{A. Vargas: Departmento de Mathem\'aticas, Universidad Aut\'onoma de  Madrid,
28049
Madrid,
Spain
}
\email{{\tt ana.vargas@uam.es}}
\urladdr{{http://matematicas.uam.es/~AFA/}}

\thanks{The first author was partially supported by the ERC grant 307617.\\
The first two authors were partially supported by the DFG grants MU 761/ 11-1 and MU 761/
11-2.\\
The third author was partially supported by grants MTM2013--40945 (MINECO) and
MTM2016-76566-P (Ministerio de Ciencia, Innovaci$\acute{\text{o}}$n y Universidades), Spain.}

\begin{abstract} 
We continue our  research on  Fourier restriction for hyperbolic surfaces, by studying local  perturbations
of the hyperbolic paraboloid $z=xy$ which are  of the form $z=xy+h(y),$ where $h(y)$ is a smooth function which is flat at the origin. The case of perturbations of finite type had already been handled before, but the flat case imposes several new obstacles. By means of a decomposition  into intervals on which $|h'''|$ is of a fixed size $\la,$ we can  apply methods devised in preceding papers,  but since we loose control on higher order derivatives  of $h$ we are forced to rework the bilinear method for wave packets that are only slowly decaying. Another problem lies in the passage from bilinear estimates to linear estimates, for which we need to  require some monotonicity of $h'''.$ 
\end{abstract}

\maketitle


\tableofcontents

\thispagestyle{empty}
\setcounter{equation}{0}
\section{Introduction}\label{intro}

In this  article we continue our study of  Fourier restriction to hyperbolic hypersurfaces that we had begun in \cite{bmv17, bmv19}.

In contrast to hyperbolic surfaces, Fourier restriction to hypersurfaces  with non-negative principal curvatures has been studied intensively by many authors (see, e.g., \cite{Bo1}, \cite{Bo2}, \cite{Bo3},\cite{MVV1}, \cite{MVV2}, \cite{TV1}, \cite{TV2},  \cite{W2}, \cite{T4},  \cite{lv10},  \cite{BoG},  \cite{bmv16}, \cite{Gu16},
\cite{Gu17}, \cite{Sto17a}). For the case of hypersurfaces of  non-vanishing  Gaussian curvature but principal curvatures of different signs, besides Tomas-Stein type Fourier restriction estimates (see, e.g.,  \cite{To},\cite{Str}, \cite{Gr},\cite{St1}, \cite{ikm}, \cite{IM-uniform},  \cite{IM}), until recently  the only case which had been studied successfully was the case of the hyperbolic paraboloid (or "saddle") in  $\R^3$:
in 2015, independently S. Lee \cite{lee05} and A. Vargas \cite{v05}  established results analogous to Tao's theorem \cite{T2} on elliptic surfaces (such as the $2$ -sphere), with the exception of the  end-point, by means of the bilinear method.  Recently, B. Stovall
\cite{Sto17b}  was able to include also the end-point case. Moreover,  C. H. Cho and J. Lee \cite{chl17}, and
J. Kim \cite{k17},  improved the range by adapting ideas by Guth  \cite{Gu16},
\cite{Gu17} which are based on the polynomial partitioning method. For further  information on the history of the restriction problem, we refer the interested reader to our previous paper \cite{bmv17}.

\medskip

We shall here study surfaces $S$ which are  local perturbations of the hyperbolic paraboloid $z=xy$ that  are given as the graph of a function
$\phi(x,y):=xy+h(y),$ where the function $h$ is smooth and satisfies
\begin{equation}\label{vanish}
h(0)=h'(0)=h''(0)=0,
\end{equation}
 i.e.,
\begin{align}\label{surface}
S:=\{(x,y,xy+h(y)):(x,y)\in \Om\},
\end{align}
where $\Om$ is a sufficiently small neighborhood of the origin. 
The Fourier restriction problem, introduced by E. M. Stein in the seventies (for general submanifolds), asks for  the range  of exponents
$\tilde p$ and $\tilde q$ for which  an a priori  estimate of the form
\begin{align*}
\bigg(\int_S|\widehat{f}|^{\tilde q}\,d\sigma\bigg)^{1/\tilde q}\le C\|f\|_{L^{\tilde
p}(\R^n)}
\end{align*}
holds  true for   every Schwartz function   $f\in\cal S(\R^3),$ with a constant $C$
independent of $f.$ Here, $d\sigma$ denotes the surface measure on $S.$

As usual, it will be more convenient to use duality and work in the adjoint setting.  If
$\cR$ denotes the Fourier restriction operator $g\mapsto \cR g:=\hat g|_S$ to the surface
$S,$ its adjoint operator $\cR^*$ is given by $\cR^*f(\xi)=\ext f(-\xi),$ where
 $\ext$ denotes the ``Fourier extension'' operator given by
\begin{align*}
	\ext f(\xi):=\widehat{f\,d\sigma}(\xi)= \int_S f(x)e^{-i\xi\cdot x}\,d\sigma(x),
\end{align*}
with  $f\in L^q(S,\sigma).$ The restriction problem is therefore
equivalent to the question of finding the appropriate range of exponents for which the
estimate
$$
\|\mathcal E f\|_{L^r(\R^3)}\le C\|f\|_{L^q(S,d\sigma)}
$$
holds true with a constant $C$ independent of the function   $f\in L^q(S,d\sigma).$

By identifying a point $(x,y)\in \Om$ with the corresponding point $(x,y,\phi(x,y))$ on $S,$
we may regard our Fourier extension operator $\ext$  as well as an operator mapping
functions on $\Om$ to functions on $\R^3,$ which in  terms of our phase function
$\phi(x,y)=xy+h(y)$ can be expressed  more explicitly  in the form
$$
\ext f(\xi)=\int_{\Om} f(x,y) e^{-i(\xi_1 x+\xi_2 y+\xi_3\phi(x,y))} \eta(x,y) \, dx dy,
$$
if $\xi= (\xi_1,\xi_2,\xi_3)\in \R^3,$ with a suitable  smooth density $\eta.$
\smallskip 

We remark that it is not really necessary to assume condition \eqref{vanish}. Indeed,  one can easily show by means of a suitable affine-linear change of coordinates that we can always remove the Taylor polynomial of degree $2$ of $h$ associated to $y=0$ from $h$  to reduce  the restriction estimates to the cases where the vanishing condition \eqref{vanish} holds true, so let us assume this condition henceforth. 

Note that if $h$ is of finite type at the origin, then this assumption implies that  there is some $m\ge 1$ such that $h^{(m+2)}(0)\ne 0,$ i.e., $h(y)=y^{m+2} a(y),$ with $a(0,0)\ne 0.$ This case had already been treated in \cite{bmv19},  so what remained open is the case where $h$ is flat at $0,$ i.e., where $h^{(k)}(0)= 0$ for every 
$k\in\N.$ 
\smallskip

Our main result, which generalizes Theorem 1.1 in \cite{bmv19},  allows to treat also the latter case under a monotonicity assumption on $h''':$

\begin{thm}\label{mainresult}
	Assume that $r>10/3$ and  $1/q'>2/r,$ and let $\ext$ denote  the Fourier extension
operator associated to the graph $S$  in \eqref{surface} of the above phase function $\phi(x,y):=xy+h(y),$ where the function  $h$ is smooth and satisfies \eqref{vanish}. Assume further that either the function $h$ is of finite type at the origin, or flat and such that  $h'''$ is  monotonic. Then, if $\Om$ is a sufficiently small neighborhood of the origin,
	\begin{align*}
		\|\ext f\|_{L^r(\R^3)} \leq C_{r,q} \|f\|_{L^q(\Om)}
	\end{align*}
	for all $f\in L^q(\Om)$.
\end{thm}

As already mentioned, the novelty of this result  lies in  the case where $h$ is flat at the origin. Assuming this, we may reduce ourselves to the region where $y>0,$ and the monotonicity assumption then  allows us to assume henceforth  without loss of generality that $h'''(y)>0$  on that region.

\medskip
Our approach does not allow to include also functions $h$ where $h'''$ admits too many oscillations, such as the function $\vp$ in Example \ref{oscfct}. 

A treatment of more general smooth perturbations $h(y)$ that are flat at the origin  imposes  indeed various serious problems, since one seems to loose control over possible bilinear estimates  when we have many  disjoint intervals on which $h'''(y)$ is of a certain size $\la,$ but  in between these intervals  
$h'''(y)$ is much bigger than $\la.$ It is even possible that the zero set of $h'''$ is a   totally disconnected Cantor type set $C$ of positive Lebesgue measure $|C|>0,$ 
and in this case it is not even clear if any kind of meaningful wave packet decomposition over this set $C$ is still  possible. Note that wave packet decompositions    play a fundamental role also in other  approaches to Fourier restriction,  based  e.g. on  multilinear methods, or polynomial partitioning techniques.
\medskip

For the proof of Theorem \ref{mainresult}, we shall build on the  methods developed in  \cite{bmv19} and \cite{bmv17}. However, serious new obstacles  do arise compared to our discussion of the finite type case in  \cite{bmv19}.

 One major difficulty  stems from the fact that our interval decomposition result in the subsequent Theorem \ref{levelset} will in fact allow us to decompose the $y$-region into intervals $I_\la$ on which the third derivative of $h$ is of a fixed dyadic size $\la>0,$  but the price that  we shall pay is that we loose any reasonable control on higher order derives of $h$  compared to the level $\la$ on these subintervals.  For the usual wave packet decompositions  of $\ext f$  this means that our wave packets will possibly  no longer be rapidly decaying away from their central tubes - indeed only a rather  slow decay can still be guaranteed.  For this reason, we shall have to rework  the usual bilinear analysis, making use  not only of the classical tubes associated to our wave packets, but also of further ``hollow tubes'' that take into account the contributions by regions  far away from the central tube.

A second major problem is that, unlike in the finite type perturbation case,  in the case of a flat perturbation  function $h$ we shall  in general no longer  be able to simply sum the contributions given  by the subintervals $I_\la.$ To resolve this issue, we shall apply a suitable bootstrap argument in Section \ref{bilinlin}, which will again make use of bilinear estimates. 

\bigskip

\noindent\textsc{Convention:}
Unless  stated otherwise, $C > 0$ will stand for an absolute constant whose value may vary
from occurrence to occurrence. We will use the notation $A\sim_C B$  to express  that
$\frac{1}{C}A\leq B \leq C A$.  In some contexts where the size of $C$ is irrelevant we
shall drop the index $C$ and simply write $A\sim B.$ Similarly, $A\lesssim B$ will express
the fact that there is a constant $C$ (which does not depend  on the relevant quantities
in
the estimate) such that $A\le C B,$  and we write  $A\ll B,$ if the constant $C$ is
sufficiently small.
\medskip



\setcounter{equation}{0}

\section{A level-size decomposition}\label{sec:levelset}
 Assume that $\varphi:I\to\R$ is a sufficiently  smooth real valued function  on some compact interval $I\subset \R.$ Our goal in  this section will be to show that we can decompose the subset of $I$  on which $\vp$ does not vanish  into countably many subintervals on each which $\varphi$ will be comparable in size to some  fixed dyadic number $\lambda$ (depending on the subinterval),  with a good control on the number of these subintervals. 
 
If $\varphi$ is of finite type near some given point, say, the origin, that is, if there is some $m\in \N$ such that $\varphi(0)=\cdots=\varphi^{(m-1)}(0)=0\neq \varphi^{(m)}(0),$ such a decomposition is quite easily established (at least on a sufficiently small neighborhood of the origin).    Indeed, a Taylor expansion shows that   $\varphi (x)\sim a x^m$ near the origin, with $a\ne 0.$  Even more, we can also control the derivatives of $\varphi:$ 
$$\varphi^{(k)}(x)\sim  a_m x^{m-k}$$
for all $k=0,\ldots,m-1$ and all $x$ near $0$. Thus, we can simply  perform a dyadic decomposition with respect to the origin, and on each of its dyadic subintervals  $\varphi$ will be comparable in size to some fixed dyadic number, and these numbers will essentially be different for different dyadic intervals.

But, if $\varphi$ is flat at the origin, that is, if $\varphi^{(k)}(0)=0$ for all $k\in \N$, we have no such easy control, since $\varphi$ and its derivatives may highly oscillate  near the origin. Let us give an example:

\begin{example}\label{oscfct}
Let $\varphi(x)=\exp(-1/x^2)\sin(1/x)$ and $\lambda>0$. Let $\epsilon_0:=\log(1/\lambda)^{-1/2}$ (so that $\exp(-1/\epsilon_0^2)=\lambda$.) We are interested in the number of connected component  intervals  into which   the level band set $U_\lambda:=\{x\in[0,1]:\lambda<\varphi(x)<4\lambda\}$ decomposes. To this end, we consider
$x\in U_\lambda\cap[\epsilon,2\epsilon]$, $\epsilon>\epsilon_0$.\\
Then $\frac1{2\epsilon}\leq\frac1x\leq\frac1\epsilon$, that means  $\sin (1/x)$ may oscillate $\sim\frac1\epsilon$ times, hence $U_\lambda\cap[\epsilon,2\epsilon]$ consists of 
$\sim\frac1\epsilon$
intervals. Dyadic summation shows that $U_\lambda$ consists of $\sim\frac1{\epsilon_0}=\sqrt{\log(1/\lambda)}$ intervals.
\end{example}

The example shows that the level band  sets  $U_\lambda$ of a flat function may consist of a large number (typically  growing  with the size of $1/\lambda$) of  intervals,    though we may  hope for some control on this  number of intervals. The following theorem does indeed provide such a control.

\begin{thm}\label{levelset}
Let $I=[a,b]$ be a compact interval and $\varphi\in C^r(I, \R)$, $r\geq1,$ and put  $C_r:=\|\varphi^{(r)}\|_{L^\infty(I)}.$ Then there exists a decomposition of $\{\varphi\neq0\}$ into pairwise disjoint intervals $I_{\lambda,\iota}$,  where $\lambda$ ranges over the set of all positive dyadic numbers $\lambda\leq \|\varphi\|_\infty,$  and where for any given $\lambda,$ the index $\iota$ is  from some index set $\mathcal I_\lambda$, such that the following hold true:
\begin{enumerate}
\item $|\I_\lambda|\leq 10r\big(1+|I|C_r^{1/r} \lambda^{-1/r}\big) \lesssim 1+\lambda^{-1/r}$.
\item For any $\lambda$ dyadic, $\iota\in\I_\lambda$ and any $x\in I_{\lambda,\iota}$ we have $\frac12\lambda <|\varphi(x)|<4\lambda.$
\end{enumerate}
\end{thm}
An immediate consequence is the following related result from  \cite{sj74}:

\begin{cor}[Sj\"olin]\label{sjoelin}
Let $I\subset \R$ be a compact interval and assume  that $\vp\in C^\infty(I,\R).$  Denote by  $E$ the set of zeros of $\vp,$  and by $\{I_k\}_k$ the component intervals of $I\setminus E.$ Then for every $\varepsilon>0$ the series $\sum\limits_k(\sup\limits_{I_k} |\vp|)^\varepsilon $ is convergent.
\end{cor} 
The proof of the theorem  will be based on the following auxiliary lemma:

\begin{lemma}\label{iteratelemma}
Let $a>0$, $b\in\R$, let $J$ be a compact interval and $\varphi\in C^r(J), \, r\ge 1.$ If there exist  points $t_0<\cdots <t_r$ in $J$ such that for every $m=0,\dots, r$ 
\begin{align}\label{signs2}
	(-1)^m\big(\varphi(t_m)-b\big)\geq a,
\end{align}
then 
\begin{align}\label{trulala}
	|J|>2 \left(\frac a{\|\varphi^{(r)}\|_\infty}\right)^{1/r}.
\end{align}
The same conclusion also holds under the assumption $(-1)^{m+1}\big(\varphi(t_m)-b\big)\geq a.$ 
\end{lemma}

\begin{proof}
By the mean value theorem, we find $t^1_m\in(t_m,t_{m+1})$, $m=0,\ldots,r-1$, with
$$\varphi'(t_m^1)(t_{m+1}-t_m)=\varphi(t_{m+1})-\varphi(t_m).$$
Then we have
\begin{align}
	(-1)^{m+1} \varphi'(t_m^1) 
	=& \frac{(-1)^{m+1}(\varphi(t_{m+1})-b)+(-1)^m(\varphi(t_m)-b)}{t_{m+1}-t_m}
	\geq\frac{2}{|J|}a.
\end{align}
Note that this condition on $\vp'$ is similar to our assumption \eqref{signs2}  on $\vp$, but now with $b=0$ and $a$ replaced by $2a/|J|.$ We can thus iterate this procedure. Formally, the following can be verified by induction:
\begin{align}
	(-1)^{m+k} \Big (\varphi^{(k)}(t_m^k)-b^{(k)}\Big)
	\geq& \left(\frac{2}{|J|}\right)^ka 
\end{align}
for certain $t_0^k,\ldots,t_{r-k}^k\in J$, and where $b^{(0)}=b$, $b^{(k)}=0$ for $k>0$. We can keep iterating as long as we can differentiate. Thus, finally we arrive for  $k=r$ at the estimate
\begin{align}
	\|\varphi^{(r)}\|_\infty \geq (-1)^{m+r} \varphi^{(r)}(t_0^r)
	\geq& \left(\frac{2}{|J|}\right)^r a ,
\end{align}
which gives \eqref{trulala}.
\end{proof}

\medskip

\noindent {\it Proof of Theorem \ref{levelset}.} 
We shall  show how to decompose the set $\{\varphi>0\}$  (assuming this set is non-empty); the set $\{\varphi<0\}$  can then be treated in a similar way. Denote by $J_l, \, l\in L,$  the connected component intervals of $\{\varphi>0\}$ ($L$ is at most countable). If $J_l$ does not intersect the boundary of $I,$ then we may write $J_l=(a_l,b_l)$, with $a_l<b_l.$ Otherwise, if $a\in J_l,$ then $\vp(a)>0,$ and  we may write $J_l=[a_l,b_l)=[a,b_l).$ Similarly, if $b\in J_l,$ then $\vp(b)>0,$ and  we may write  $J_l=(a_l,b_l]=(a_l,b].$ 
\smallskip

In a second step,  given $l\in L, $ we  shall further decompose the interval $J_l.$ To this end,  we choose a point $c_0^l\in J_l$  so that $\varphi(c^l_0)$ is a dyadic number. This is always possible, unless  $J_l$ intersects the boundary and there is some dyadic number $\la$ such that $\la<\vp(x)<2\la$ for every $x\in J_l.$ But then $J_l$ does already have the desired property, and since there are at most two  intervals $J_l$ of that type, we may ignore these cases in the sequel. Our aim is to construct subintervals of $J_l$ which are as large as possible  so that  $\varphi$ will change by a dyadic number on them.
\smallskip

Accordingly, we choose  $c_1^l>c^l_0$ maximal in $(a_l, b_l]$ such that $\varphi(c_{1}^l)\in[\tfrac12\varphi(c^l_0),2\varphi(c^l_0)]$. 
Then, recursively, if we have already constructed $c_{0}^l,\ldots,c_k^l$,  we choose $c_{k+1}^l:=\max\{x\in(c_k^l,b_l]:\tfrac12\varphi(c_k^l)\leq\varphi(x)\leq2\varphi(c_k^l)\}$, so that $\varphi(c_{k+1}^l)\in\{\tfrac12\varphi(c^l_k),2\varphi(c^l_k)\}$, or $c^l_{k+1}=b$. The latter case can  only happen if $J_l=(a_l,b]$ and $\varphi(b)>0$, in which case $\varphi(b)$ may or may not be a dyadic number. If $c^l_{k+1}=b$, we stop the construction; otherwise, we get a countable sequence of 
points $c_k^l, k\in \N.$  We claim that in the latter case,  $\lim\limits_{k\to\infty}c_k^l=b_l$: 

Obviously $c_k^l$, $k\to\infty$, is a convergent sequence, on which $\varphi$ assumes  only dyadic values, and by construction we have 
\begin{align}\label{16oct1129}
	\max\{\varphi(c_k^l),\varphi(c_{k+1}^l)\}\leq 2|\varphi(c_k^l)-\varphi(c_{k+1}^l)| \leq 2\|\varphi'\|_\infty|c_k^l-c^l_{k+1}|.
\end{align}
Therefore $\lim\limits_{k\to\infty}\varphi(c_k^l)=0$, which implies the claim. 

We construct points $c_k^l$ with indices  $k<0$ and $c_0^l>c_{-1}^l >c_{-2}^l>\cdots$ in an analogous manner  by ''moving to the left'', and obtain that$\lim\limits_{k\to-\infty}c_k^l=a_l$ in case where we do not reach $a_l=a$ in finitely many steps. In this way we obtain a measure-theoretic disjoint decomposition
$$
J_l=\bigcup\limits_{k\in Z^l}[c_k^l,c^l_{k+1}],
$$
where $Z^l\subset \Z$  is of the form $Z^l=\{k\in \Z: K^l_-< k< K^l_+\},$  $-\infty\le K^l_-<K^l_+\le +\infty.$

Recall that  then $\varphi(c_k^l)$ and $\varphi(c^l_{k+1})$ are dyadic numbers which differ exactly by a factor 2, with the  possible  exception of  the cases  where $c_k^l=a,$ or $c^l_{k+1}=b$.

We sort the points $c_k^l$ according to which dyadic value $\varphi(c_k^l) $ takes. For bounding the number of component intervals $[c_k^l,c^l_{k+1}],$  we may ignore the intervals at the boundary of $I$ containing $a,$ respectively $b,$ since there are at most two of them. For any dyadic parameter $\lambda>0$, let 
\begin{align}\label{alambda}
	A(\lambda):=\big\{c_k^l: \{\varphi(c_k^l),\varphi(c_{k+1}^l)\}=\{\lambda,2\lambda\}\big\}.
\end{align}
Note that by construction, for any $c_k^l\in A(\lambda)$ and $c_k^l<x<c_{k+1}^l,$ we have 
\begin{align}
	\lambda/2<\varphi(x)<4\lambda.
\end{align}

Therefore the sets $I_{\lambda,\iota}:=[c_k^l,c_{k+1}^l)$, indexed by $\iota=c_k^l\in A(\lambda)$, together with a similar decomposition of $\{\varphi<0\}$ and possibly two boundary intervals, give  the desired decomposition. 

For the bound in (i), it is then enough to show that
\begin{align}
	|A(\lambda)|\leq4r\big(1+|I|C_r^{1/r} \lambda^{-1/r}\big).
\end{align}

Using \eqref{16oct1129}, we see that for $c_k^l\in A(\lambda)$ we have
$$\lambda\leq 2\|\varphi'\|_\infty|c_k^l-c^l_{k+1}|.
$$
This implies that $A(\lambda)$ is finite (and in particular, that the number of connected components $J_l$ of $I$ is indeed countable). More precisely, we obtain the bound 
$$|A(\lambda)|\leq 2\|\varphi'\|_\infty |I|\lambda^{-1},$$
which already gives the bound in (i) when  $r=1$. 
\medskip

When $r\ge 2,$ we have to make use of  estimates of the number of oscillations of certain sizes of the  function $\vp$ and its derivatives  and apply Lemma \ref{iteratelemma}. To this end, 
 we fix  now a dyadic parameter $\lambda,$ and order the  elements of $A(\lambda)$ by $y_1<\cdots<y_M,$
where $M:=|A(\lambda)|.$

We claim that we have 
\begin{align}\label{claim1}
	\{\varphi(y_m),\varphi(y_{m+1})\}=\{\lambda,2\lambda\} \qquad\text{ for all $m=1,\ldots,M-1$}.
\end{align}
We know by \eqref{alambda} that the  $\varphi(y_m)$ can only take the values $\lambda,$ or $2\lambda$.
Given $y_m=c_k^l$ for some $k$ and $l$, we assume without loss of generality  that $\varphi(y_m)=2\lambda,$ and hence $\varphi(c_{k+1}^l)=\lambda$ (the case where $\varphi(y_m)=\lambda$  can be treated in a similar way). 

 If we had  $\varphi(y_{m+1})=2\lambda$, then we  would choose $c_{k'}^{l'}> c_{k+1}^l$ minimal such that $\varphi(c_{k'}^{l'})\geq2\lambda$. Then necessarily  $\varphi(c_{k'-1}^{l'})<2\lambda,$ and since this  is a dyadic number, even $\varphi(c_{k'-1}^{l'})\leq\lambda$. But then we would have
$$2\lambda\leq \varphi(c_{k'}^{l'}) \leq 2\varphi(c_{k'-1}^{l'}) \leq 2\lambda,$$
so that all the above inequalities must be equalities and hence $c_{k'-1}^{l'}\in A(\lambda)$ while $y_m<c^{l'}_{k'-1}<c^{l'}_{k'}\leq y_{m+1}$,  where the three  points $y_m, c^{l'}_{k'-1}$ and $y_{m+1}$  are in
 $A(\la).$ This would lead to a contradiction.
Therefore $\varphi(y_{m+1})=\lambda$, which verifies \eqref{claim1}.
\medskip

Let us assume for the sake of concreteness that, say, $\varphi(y_1)=\lambda$ (the proof  will be quite similar when $\varphi(y_1)=2\lambda$). Then it is easy to see that \eqref{claim1} implies that for every $m$ 
\begin{align*}
	(-1)^m\Big(\varphi(y_m)-\frac32\lambda\Big)=\frac \la 2.
\end{align*}

 In the next step, we organize our sequence $\{y_m\}$ into groups of $r+1$ points: We may assume without loss of generality that $M=|A(\lambda)|\geq4r$. Let 
 $$
 I_j:=[y_{j+1},y_{j+r+1}], \qquad j\in(r+1)\N, \  \frac{j}{r+1}\leq\frac{M}{r+1}-1.
 $$
 This  means  we have at least $\frac M{r+1}-1\geq\frac{M}{4r}$ pairwise disjoint intervals $I_j$ such that in each $I_j$ we have points $y_{j+1+m},\, m=0,\ldots, r,$ with
\begin{equation*}
(-1)^{m+j+1}\Big(\varphi(y_{j+1+m})-\frac32\lambda\Big)=\frac \la 2.
\end{equation*}
	We can thus apply Lemma \ref{iteratelemma} and find that
$$
 |I_j|\geq 2\left(\frac{\lambda}{2C_r}\right)^{1/r}\geq \left(\frac{\lambda}{C_r}\right)^{1/r}.
$$
We conclude that
$$
|I|\geq \sum_j|I_j| \geq \frac{M}{4r}\left(\frac{\lambda}{C_r}\right)^{1/r},
$$
and hence 
$$M\leq 4r|I|C_r^{1/r}\lambda^{-1/r}.
$$
\qed


\setcounter{equation}{0}
\section{Fourier restriction over rectangles on which $h'''\sim \la$}\label{h'''simla}
In the sequel, we shall always assume that $h$ is flat at the origin. By applying Theorem \ref{levelset} to $\vp:=h''',$  we are led to first restricting ourselves to intervals $I_{\la,\iota}$ on which we may assume that $h'''(x)\sim \la,$ where $\la>0$ is a fixed dyadic number. Note also that by choosing $\Om$ sufficiently small in Theorem \ref{mainresult}, we may assume that 
\begin{equation}\label{la1}
0<\la <\la_1, \quad\text{where}\ \la_1\ll 1\quad \text{is sufficiently small}.
\end{equation}
Let us therefore now assume that on some interval $I=(c_I-d_I,c_I+d_I)$ of positive length $2d_I>0$ contained in $(-1,1)$ we know that 
$$
h'''(y) \sim \la.
$$
Then a standard change of coordinates argument (cf. \cite{bmv19}) allows to pass essentially to a situation where $c_I=0,d_I=1,$ and where \eqref{vanish} holds true again. Indeed, write $y=c_I+d_Iy',$ and consider $\phi(x,y)=xy+h(y).$ Changing coordinates $y=c_I+d_Iy',$  a Taylor expansion of $h$ then leads to 
\begin{eqnarray*}
\phi(x,y)= x(c_I+d_Iy') +h(c_I)+h'(c_I) d_I y'+ \frac{h''(c_I)}2 d_I^2 (y')^2 + d_I H(y'),
\end{eqnarray*}
where $H^{(k)}(0)=0$ for $k=0,1,2,$ and $H'''(y')\sim d_I^2\la=:\epsilon.$ Thus
\begin{eqnarray*}
\phi(x,y)/d_I&=& y'(x+\frac{h''(c_I)}2 d_I y') +H(y')+ \text{affine linear terms}\\
 &=&x'y'+H(y')+ \text{affine linear terms},
\end{eqnarray*}
if we put $x':=x+\frac{h''(c_I)}2 d_I y'.$  Let us therefore define $\tilde \phi(x',y'):=x'y'+H(y').$ Then $H'''(y')\sim \epsilon,$ where $0<\epsilon\ll 1.$ 
Thus, if we  put $\Om_I:=\Om\cap\{y\in I\}$ and denote by 
$$
\ext_I f(\xi)=\int_{\Om_I} f(x,y) e^{-i(\xi_1 x+\xi_2 y+\xi_3\phi(x,y))} \eta(x,y) \, dx dy
$$
 the contribution of the $y-$ Interval $I$ to  to $\ext f,$ and define correspondingly
$$
\tilde \ext_I f(\xi)=\int_{(-1,1)^2} f(x',y') e^{-i(\xi_1 x'+\xi_2 y'+\xi_3\tilde\phi(x',y'))} \tilde\eta(x',y') \, dx' dy',
$$
with a suitable cutoff function $\tilde \eta,$ 
then an  easy scaling argument shows  that  the following estimates for $\ext_I$ and
$\tilde\ext_I$ are equivalent:
\begin{equation}\label{extI1}
\|\tilde\ext_I f\|_{L^r}\le C \|f\|_{L^q};
\end{equation}
\begin{equation}\label{extI2}
\|\ext_I g\|_{L^r}\le C d_I^{1-2/r-1/q} \|g\|_{L^q}
\end{equation}
for all g with $\supp g\subset I $ (and support in $x$ sufficiently small).

 It will therefore suffice to prove  estimate \eqref{extI1}, and to this end recall that  $H(0)=H'(0)=H''(0)=0$ and $H'''(y)\sim \epsilon\ll1.$ Taylor expansions then show that $\|H\|_{C^2}\lesssim \epsilon.$  We may thus re-write $H=\epsilon h,$  i.e., 
 \begin{equation}\label{phicub}
\phi(x,y)=xy+\epsilon h(y), \quad  \epsilon\ll1 \qquad ((x,y)\in Q:=[-1-1]^2),
\end{equation}
 with a new function $h$ satisfying the following conditions: 
 \begin{equation}\label{cubtype}
 h(0)=h'(0)=h''(0)=0 \quad \text{and}\quad  {C_3}/4\le h'''(y)\le C_3 \quad \text{for all}\  |y|\le 1,
\end{equation}
where the constant $C_3$ will be  assumed to be fixed constant $C_3>0.$ Note that \eqref{cubtype} implies that 
$$
\|h\|_{C^2}\lesssim C_3.
$$

We like to stress the point that we cannot assume any reasonable control on derivatives of order $4$ or higher of $h.$ We therefore call such a function  $h$ of $y$  a  perturbation  function {\it  of coarse  cubic type} of the  phase $xy,$ in contrast to the notion of a perturbation of cubic type defined by (3.1) in  \cite{bmv19}, where we had in addition also required a suitable control on higher order derivatives.
\smallskip

A major goal of this paper will be to prove the following uniform restriction estimate for perturbations of coarse cubic type  of the parabolic hyperboloid: 
\begin{thm}\label{cubrest}
Assume that $\phi$ is given by \eqref{phicub} on the cube $Q,$ where $h(y)$ is a perturbation of coarse cubic type,  let  $S$ be given as the graph of $\phi.$  By $\ext$ we denote  again the Fourier extension operator associated to $S.$ 
Assume further that $r>10/3$ and  $1/q'>2/r.$ Then
$$
\|\ext f\|_{L^r(\R^3)} \leq C_{r,q} \|f\|_{L^q(Q)} \qquad\text{for all} \ f\in L^q(Q),
$$
where the constant $C_{r,q}$ may depend on the constant $C_3$ in \eqref{cubtype}, but not on any further property of $h,$ and not on $\epsilon.$ 
\end{thm}
The proof will be based on a modification of  the bilinear method, taking into account the slow decay of wave packets that we are available only in this context. We shall therefore largely follow our approach from  \cite{bmv19} only mostly concentrate on those points of the arguments which will require new ideas.
  \medskip

\subsection{Admissible pairs of sets $U_1,$ $U_2$ on which transversalities are of a fixed
size}\label{pairs of sets}

Recall  that the bilinear approach is based on  bilinear estimates of the
form
\begin{align}\label{bil1}
 \|\ext_{U_1}(f_1)\,\ext_{U_2}(f_2)\|_p \leq C(U_1,U_2) \|f_1\|_2\|f_2\|_2.
\end{align}
 Here,  $\ext_{U_1}$ and $\ext_{U_2}$ are the Fourier extension operators associated to
 patches of sub-surfaces $S_i:=\graph \phi|_{U_i}\subset S,\ i=1,2,$  with $U_i\subset \Om.$
 What is crucial for obtaining  useful bilinear estimates is that the two  patches of
 surface $S_1$ and $S_2$ satisfy certain {\it transversality conditions,} which  are
 stronger than  just assuming that $S_1$ and $S_2$ are transversal as hypersurfaces (i.e.,
 that all normals to $S_1$ are transversal to all normals to $S_2$). Indeed, what  is
 needed in addition is the following (cf. \cite {bmv17},\cite{lee05}, \cite{v05}, \cite{lv10}, or \cite{be16}):
 \smallskip

Denoting by  $H\phi$ the Hessian of $\phi,$ we consider the following quantity
\begin{align}\label{transs}
\tilde\Gamma^\phi_{z}(z_1,z_2,z_1',z_2'):=	\left\langle
(H\phi)^{-1}(z)(\nabla\phi(z_2)-\nabla\phi(z_1)),\nabla\phi(z_2')-\nabla\phi(z_1')\right\rangle.
\end{align}
 If its modulus  is  bounded from below by a constant $c>0$  for all $z_i=(x_i,y_i),\, z_i'=(x_i',y_i')\in U_i$, $i=1,2$, $z=(x,y)\in U_1\cup U_2,$ then we have \eqref{bil1} for $p>5/3,$ with a constant $C(U_1,U_2)$ that depends only on this constant  $c$  and on upper bounds for the derivatives of $\phi.$  If $U_1$ and
$U_2$ are sufficiently small (with sizes depending on upper bounds of the first and second
order derivatives of $\phi$ and a lower bound for the determinant of $H\phi$) this
condition  reduces to the estimate
\begin{align}
|\Gamma^\phi_{z}(z_1,z_2)|\geq c,
\end{align}
for $z_i=(x_i,y_i)\in U_i$, $i=1,2$, $z=(x,y)\in U_1\cup U_2$, where
\begin{align}\label{trans}
\Gamma^\phi_{z}(z_1,z_2):=	\left\langle
(H\phi)^{-1}(z)(\nabla\phi(z_2)-\nabla\phi(z_1)),\nabla\phi(z_2)-\nabla\phi(z_1)\right\rangle.
\end{align}
It is easy to check that for $\phi(x,y)=xy+\epsilon h(y)$, we have
\begin{eqnarray} \label{gammaz}
\Gamma^\phi_{z}(z_1,z_2)
	 &=:& 2(y_2-y_1)\,\tau_{z}(z_1,z_2),
\end{eqnarray}
where
\begin{equation}\label{TV1}
\tau_{z}(z_1,z_2):=x_2-x_1+\epsilon[h'(y_2)-h'(y_1)-\frac 12 h''(y)(y_2-y_1)].
\end{equation}
As in \cite{bmv19}, it will be particularly important to look at the
expression
\eqref{TV1} when $z=z_1\in U_1,$ and $z=z_2\in U_2,$  so that the two ``transversalities''
\begin{align}\label{TV2}
	\tau_{z_1}(z_1,z_2)=x_2-x_1+\epsilon[(h'(y_2)-h'(y_1)-\frac 12 h''(y_1)(y_2-y_1)]\\
	\tau_{z_2}(z_1,z_2)=x_2-x_1+\epsilon[(h'(y_2)-h'(y_1)-\frac 12 h''(y_2)(y_2-y_1)] \label{TV3}
\end{align}
become relevant.
Note  the following relation  between these quantities:
\begin{align}\label{TV2+TV3}
|\tau_{z_1}(z_1,z_2)-\tau_{z_2}(z_1,z_2)|&=\frac \epsilon 2 |h''(y_2)-h''(y_1)||y_2-y_1|\nonumber
\sim \eps  |h'''(\eta)|(y_2-y_1)^2\\
&\sim \eps (y_2-y_1)^2,
\end{align}
where $\eta$ is some intermediate point.
\medskip

Following Section 3 in \cite{bmv19}, we shall try to devise neighborhoods $U_1$ and $U_2$ of two given points
 $z_1^0=(x_1^0,y_1^0)$ and $z_2^0=(x_2^0,y_2^0)$ on which these quantities are roughly
  constant for $z_i=(x_i,y_i)\in U_i,$ $i=1,2$,
 and which  are also  essentially chosen as large as possible. The corresponding pair $(U_1,U_2)$ of
 neighborhoods of $z^0_1$ respectively $z^0_2$ will be called an {\it admissible pair}. For a motivation of the precise definition of admissible pairs that we shall give in the next subsection, we refer to \cite{bmv19}.

\subsection{Definition of admissible pairs within $Q\times Q$}\label{preciseadmissible}

To begin with,  we fix a large dyadic constant $C_0\gg1.$

\smallskip
In  a first step,  we   perform a  classical {\bf dyadic decomposition in the
$y$-variable}:

For a given  dyadic number $0<\rho\lesssim 1,$  we denote for
$j\in\mathbb Z$ such that $|j|\rho\le 1$  by $I_{j,\rho}$ the dyadic interval
$I_{j,\rho}:=[j\rho,j\rho+\rho)$ of length $\rho,$ and by $V_{j,\rho}$ the corresponding
horizontal ``strip''   $V_{j,\rho}:=[-1,1]\times I_{j,\rho}$ within $Q.$
Given two dyadic intervals $J,\,J'$ of the same size, we say that they are
{\it related} if their parents are adjacent but they are not adjacent. We divide each
dyadic interval $J$ in a disjoint union of dyadic subintervals $\{I_J^k\}_{1\le k\le
C_0/8},$ of length  $8|J|/C_0.$  Then, we define $(I,I')$ to be an \it
admissible pair of dyadic intervals \rm if and only if there are $J$ and $J'$ related
dyadic intervals and $1\le k,\,j\le C_0/8$ such that $I=I_J^k$ and $I'=I_{J'}^j.$

We say that a pair of strips $(V_{j_1,\rho},V_{j_2,\rho})$ is {\it admissible } and  write
$V_{j_1,\rho}\backsim V_{j_2,\rho},$ if  $(I_{j_1,\rho},I_{j_2,\rho})$ is a pair of
admissible dyadic intervals. Notice that in this case,
  \begin{align}\label{admissibleV}
C_0/8< |j_2-j_1|< C_0/2.
\end{align}
One can  easily see that  this leads to the following disjoint decomposition of $Q\times
Q:$
\begin{align}\label{whitney1}
Q\times Q= \overset{\cdot}{\bigcup\limits_{\rho}}
\,\Big(\overset{\cdot}{\bigcup\limits_{V_{j_1,\rho}\backsim
V_{j_2,\rho}}}V_{j_1,\rho}\times V_{j_2,\rho}\Big),
\end{align}
where the first union is meant to be over all such dyadic $\rho$'s.

\medskip
In a second step, we perform a non-standard {\bf Whitney type decomposition of any given
admissible pair of strips}, to obtain subregions in which the transversalities are roughly
constant.

To simplify notation, we fix $\rho$ and an admissible pair  $(V_{j_1,\rho},V_{j_2,\rho}),$
and simply write $I_i:=I_{j_i,\rho},\, V_i:=V_{j_i,\rho}, \, i=1,2,$
so that $I_i$ is an interval of length $\rho$ with left endpoint  $j_i\rho,$   and
\begin{align}\label{Vi}
V_1=[-1,1]\times I_1, \qquad V_2=[-1,1]\times I_2,
\end{align}
are rectangles of dimension $2\times \rho,$ which are vertically  separated at scale
$C_0\rho.$
More precisely, for $z_1=(x_1,y_1)\in V_1$ and $z_2=(x_2,y_2)\in V_2$ we have
$|y_2-y_1|\in
|j_2\rho-j_1\rho|+[-\rho,\rho],$ i.e.,
\begin{align}\label{yseparation}
	C_0\rho/2\le |y_2-y_1|\le  C_0\rho.
\end{align}

Let $0<\delta\lesssim \eps^{-1}\rho^{-2}$ be a dyadic number (note that $\delta$ could be big,
depending on $\rho$), and  let $\J$ be  the set of points
which partition the interval $[-1,1]$ into (dyadic) intervals of the same length  $\eps\rho^2\de.$

\smallskip
Similarly, for $i=1,2,$ we choose a finite  equidistant partition $\I_i$  of  width
$\rho(1\wedge\delta)$  of  the interval $I_i$  by  points $y_i^0\in \I_i.$
Note: if $\de>1,$ then $\rho(1\wedge\delta)=\rho,$ and we can choose for $\I_i$ just
 the
 singleton  $\I_i=\{y_i^0\},$ {where $y_i^0$ is the left endpoint of $I_i.$}

\begin{defn}
For any parameters $x^0_1,t^0_2\in\J,$  $y^0_1\in\I_1$  defined in the previous lines and $y^0_2$ the left
endpoint of $I_2,$ we  define the
sets
\begin{align}
U_1^{x^0_1,y_1^0,\delta} :=\{(x_1,y_1)&:0\le y_1-y_1^0< \rho(1\wedge\delta),\,
0\le x_1-x^0_1+\eps\tfrac{h''(y_1^0)}{2}(y_1-y_1^0)< \eps\rho^2\delta \}, \nonumber\\
\label{whitneybox1}&\\
U_2^{t^0_2,y_1^0,y^0_2,\delta} :=\{(x_2,y_2)&:0\le y_2-y^0_2<\rho, \nonumber\\&\quad0\le x_2-t^0_2+\eps[h'(y_2)-h'(y_1^0)-\tfrac{h''(y_1^0)}{2}(y_2-y_1^0) ]<
\eps\rho^2\delta\},\nonumber
\end{align}
and the points
\begin{equation}\label{pointsinU}
z^0_1=(x^0_1,y^0_1), \qquad z^0_2=(x^0_2,y^0_2)
\end{equation}
where \begin{equation*}
x_2^0:=t^0_2-\eps[h'(y^0_2)-h'(y_1^0)-\tfrac{h''(y_1^0)}{2}(y^0_2-y_1^0) ].
\end{equation*}
\end{defn}

Observe that then
$$
z^0_1\in U_1^{x^0_1,y_1^0,\delta}\subset V_1\quad \text{  and   } \quad z^0_2\in
U_2^{t^0_2,y_1^0,y^0_2,\delta}\subset V_2.
$$
Indeed, $z^0_i$ is in some sense the ``lower left'' vertex of $U_i,$ and the horizontal projection of
$U_2^{t^0_2,y_1^0,y^0_2,\delta}$ equals $I_2.$
Note that if we define $a^0$ by 
\begin{align}\label{a0defi}
	a^0:=\tau_{z_1^0}(z_1^0,z_2^0),
\end{align} 
then $x_1^0+a^0=t^0_2.$ 
Notice also that we may re-write
\begin{equation}\label{U2alt}
U_2^{t_2^0,y_1^0,y_2^0,\delta}=\{z_2=(x_2,y_2): 0\le\tau_{z_1^0}(z_1^0,z_2)-a^0<\eps\rho^2\delta, \ 0\le y_2-y_2^0<\rho\}.
\end{equation}
In particular, $U_1^{x^0_1,y_1^0,\delta}$ is   essentially a parallelepiped of side lengths $\sim
 \eps\rho^2\de\times \rho(1\wedge \de),$ containing the point $(x^0_1,y_1^0),$  whose
longer side has slope $y_1^0$ with respect to the $y$-axis.

Moreover, if $\de\ll 1,$ then $U_2^{t^0_2,y_1^0,y^0_2\delta}$ is a thin curved box of width $\sim\eps\rho^2 \de$ and
length $\sim\rho,$  contained in a rectangle of dimension $\sim \rho^2\times \rho$ whose axes are
parallel to the coordinate axes (namely the part of a $\rho^2\delta$-neighborhood of a curve  of curvature $\sim\eps$ containing the
point $(x^0_1,y^0_1)$ which lies  within the horizontal strip  $V_2$).  The case $\de\ll 1$ will  therefore be called the {\bf curved box case}.

  If $\de\gtrsim 1,$ then
$U_2^{t^0_2,y_1^0,y^0_2\delta}$ is essentially  a rectangular  box of dimension $\sim \eps\rho^2\de\times \rho$ lying
in the same horizontal strip. The case $\de> 1$ will  therefore be called the {\bf straight  box case}.
\smallskip

Note also that we have chosen to use  the parameter $t^0_2$ in place of using $x^0_2$ here, since with this
choice by \eqref{TV1} the identity
\begin{equation}\label{t2mean}
\tau_{z_1^0}(z_1^0,z_2^0)=t^0_2-x_1^0
\end{equation}
holds true, which will become quite useful in the sequel. We next have to relate the parameters $x_1^0,t^0_2,
y_1^0,y_2^0$ in order to give a precise definition of an admissible pair.

\smallskip
 Here, and in the sequel, we shall always assume that the points $z_1^0,z_2^0$ associated to these parameters
 are given by \eqref{pointsinU}.

\smallskip

\begin{defn}
Let us call a pair
 $(U_1^{x_1^0,y_1^0,\delta},U_2^{t_2^0,y_1^0,y_2^0,\delta})$ an {\it admissible
 pair of type 1
 (at scales $\de,\;\rho$ and contained in $V_1\times V_2$),} if the following two conditions hold
 true:
\begin{align}\label{admissible1}
\frac {C_0^2}4\eps\rho^2\de\le
|\tau_{z_1^0}(z_1^0,z_2^0)|&=|t_2^0-x_1^0|<4 \,C_0^2\eps\rho^2\de,\\
\frac{C_0^2}{512}\eps\rho^2(1\vee
\de)\le|\tau_{z_2^0}(z_1^0,z_2^0)|&< 5\,
C_0^2\eps\rho^2(1\vee \de).\label{admissible2}
\end{align}
By $\cP^{\de}$ we shall denote the set of all admissible pairs of
type 1 at scale $\de$ (and  $\rho$, contained in $V_1\times V_2,$), and by $\cP$
the  corresponding union over all dyadic
scales $\de.$
\end{defn}

Observe that, by \eqref{TV2+TV3}, we have
$\tau_{z_2^0}(z_1^0,z_2^0)-\tau_{z_1^0}(z_1^0,z_2^0)\sim\eps(y_2^0-y_1^0)^2.$
In view of   \eqref{admissible1} and \eqref{yseparation} this shows that condition
\eqref{admissible2} is automatically satisfied, unless $\de\sim 1.$

We remark that it would indeed be more appropriate to denote the sets $\cP^{\de}$ by $\cP^{\de}_{V_1\times V_2},$  but we
want to simplify the notation. In all instances in the rest of the paper $\cP^{\de}$ will be  associated  to a
fixed admissible pair of strips $(V_1,V_2),$ so  that our imprecision will not cause any ambiguity.
The next lemma can be proved by closely following the arguments in the proof of the corresponding Lemma 2.1 in \cite{bmv17} and just using \eqref{cubtype}.
\color{black}

\begin{lemma}\label{sizeofdeltas}
If  $(U_1^{x_1^0,y_1^0,\delta},U_2^{t_2^0,y_1^0,y_2^0,\delta})$ is an admissible
pair of type 1, then for
all  $(z_1,z_2)\in (U_1^{x_1^0,y_1^0,\delta},U_2^{t_2^0,y_1^0,y_2^0,\delta})$ ,
$$
|\tau_{z_1}(z_1,z_2)|\sim_{8} C_0^2 \eps\rho^2\de\mbox{   and   }
|\tau_{z_2}(z_1,z_2)|\sim_{1000} C_0^2 \eps\rho^2(1\vee\de).
$$
\end{lemma}

\medskip
Up to now we focused on the case $|\tau_{z_1}(z_1,z_2)|\lesssim|\tau_{z_2}(z_1,z_2)|.$
For the symmetric case, corresponding to the situation where $|\tau_{z_1}(z_1,z_2)|\gtrsim|\tau_{z_2}(z_1,z_2)|$, by
interchanging the roles of $z_1$ and $z_2$  we define
accordingly  for any $t^0_1, x^0_2\in\J,$  $y^0_1$ the left endpoint of $I_1$ and $y^0_2\in\I_2$  the sets $\widetilde{U}_1^{t_1^0,y_1^0,y_2^0,\delta}$ and $\widetilde{U}_2^{x_2^0,y_2^0,\delta}$
in analogy with our discussion in \cite{bmv17}, and denote the  corresponding admissible pairs  $(\widetilde{U}_1^{t_1^0,y_1^0,y_2^0,\delta}, \widetilde{U}_2^{x_2^0,y_2^0,\delta})$ as {\it admissible pairs of type 2}. We shall skip  the details.\\
By $\tilde \cP^{\de},$ we shall denote the set of all admissible pairs of
 type 2  at scale $\de$  (and $\rho$, contained in $V_1\times V_2,$), and by $\tilde \cP$
the  corresponding unions over all dyadic
scales $\de.$

\smallskip
In analogy with Lemma \ref{sizeofdeltas}, we have
\begin{lemma}\label{sizeofdeltas2}
If $(\tilde U_1,\tilde U_2)=({\tilde U}_1^{t^0_1,y_1^0,y_2^0,\delta},{\tilde U_2}^{x^0_2,y_2^0,\delta})\in \tilde \cP^\de$  is an admissible pair of type 2, then for all $(z_1,z_2)\in(\tilde U_1,\tilde U_2)$ we have
$$
|\tau_{z_1}(z_1,z_2)|\sim_{1000} C_0^2 \eps\rho^2(1\vee\de)\mbox{   and   }
|\tau_{z_2}(z_1,z_2)|\sim_8 C_0^2 \eps\rho^2\de.
$$
\end{lemma}

The crucial bilinear estimates that we shall establish for admissible pairs are given in the following theorem, which extends  Theorem 4.3 in  \cite{bmv19} to the case of coarse cubic perturbations $h(y):$ 

\begin{thm}\label{bilinear2}
Let $p>5/3,$ $q\ge2.$ Then, for every admissible pair  $(U_1,U_2)\in \cP^\de $
at scale $\de,$  the following bilinear estimates hold true:
If $\de>1$ and $\epsilon \de\rho^2\le 1$, then
  \begin{align*}
	 \|\ext_{U_1}(f)\ext_{U_2}(g)\|_p
	\leq C_{p,q}  (\epsilon\delta\rho^3)^{2(1-\frac 1p-\frac 1q)}\|f\|_q\|g\|_q.
 \end{align*}
If  $\de\le 1,$ then
 \begin{align*}
	 \|\ext_{U_1}(f)\ext_{U_2}(g)\|_p
	\leq C_{p,q}\, \,(\epsilon\rho^3)^{2(1-\frac 1p-\frac 1q)} \,  \delta^{5-\frac 3q-\frac 6p}\|f\|_q\|g\|_q.
 \end{align*}
The  constants in these estimates are independent of the given admissible pair, of $\eps, \rho$ and of $\de.$ The same
 estimates are valid for admissible pairs
 $(\tilde U_1,\tilde U_2)\in \tilde\cP^\de $ of type 2.
\end{thm}

The case $\de\le 1$ of this theorem can easily be reduced by means of suitable coordinate transformations to the  main result in the next subsection on  ``prototypical admissible  pairs'' (compare \cite{bmv19}.)  Its proof will require new techniques. 

The case where $\de>1,$ which is in some sense more classical and easier to deal with, can be handled by combining the  new techniques that we shall explain for the case $\de\le1$ with more classical methods as described in Subsection 4.2 of \cite{bmv19}. We shall leave the details of the  treatment of this case to the interested reader. 

\medskip
Given the bilinear estimates of Theorem \ref{bilinear2}, we can then exactly follow or arguments from Section 5 in \cite{bmv19} in order to pass from these bilinear estimates to the linear Fourier extension estimate in Theorem \ref{cubrest}  and in this way complete the proof of Theorem \ref{cubrest}. Indeed, one easily checks that all  arguments in 
Section 5 of  \cite{bmv19} effectively only require a control on the first three derivatives of $F,$ but not on higher order derivatives, as had already been observed in  Section 5 of \cite{bmv19}, so that these arguments work as well for  perturbations  of coarse cubic type $h$ in place of cubic type perturbations.

 \subsection{A prototypical admissible pair in the curved box case and the crucial scaling
transformation}\label{proto}

In this section we shall present a  \lq\lq prototypical"
case where  $U_1$ and $U_2$ will form an admissible pair of type 1 centered at  $z_1^0=0\in U_1$ and $z_2^0\in
U_2,$  with  $\eps\sim1,\rho\sim1$ and $\delta\ll1$, i.e., $|y_1^0-y_2^0|\sim 1,$ and $|\tau_{z_2^0}(z_1^0,z_2^0)|\sim 1$  but
$|\tau_{z_1^0}(z_1^0,z_2^0)|\sim \delta\ll 1.$  This means that we shall be in the curved box case.

Fix  a small number $0<c_0\ll1$ ($c_0=10^{-10}$ will, for instance, work). Assume that
$0<\delta\le1/10,$ and put
\begin{align}
	U_1:=&[0,c_0^2\delta)\times [0,c_0\delta) \label{U1}\\
	U_2:=&\{(x_2,y_2):0\le y_2-b< c_0,0\le x_2+F'(y_2)-a<c_0^2\delta\},\label{U2}
\end{align}
where $|b|\sim_21$, $|a|\sim_4\delta$ and $F$ is a function of {\it coarse cubic type} in the sense of \eqref{cubtype}, i.e.,
\begin{eqnarray}\label{cuberrorF}
F(0)=F'(0)=F''(0)=0 \quad \text{and}\quad  {C_3}/4\le F'''(y)\le C_3 \quad \text{for all}\  y.
\end{eqnarray}

\noindent\bf Remark. \rm Note that in the case $\eps=1$, if we set $C_0=1/{c_0},$  $\rho=c_0,$ then any admissible pair
$(U_1,U_2)=(U_1^{0,0,\delta},U_2^{a,0,b,\delta})$, as in \eqref{whitneybox1}, would satisfy \eqref{U1} and \eqref{U2} with the above conditions on $a$ and $b$ and suitable $F$.
\medskip

Our bilinear result in this prototypical case  is as follows:
\begin{thm}[prototypical case]\label{bilinear}
Let $p>5/3,$ and let  $U_1,U_2$ be as in \eqref{U1}, \eqref{U2}. Assume further that  $\phi(x,y)= xy+F(y),$   where $F$ is a real-valued  smooth perturbation function of cubic type, i.e., satisfying estimates \eqref{cuberrorF}, and denote by
$$
\ext_{U_i} f(\xi)=\int_{U_i} f(x,y) e^{-i(\xi_1 x+\xi_2 y+\xi_3\phi(x,y))} \eta(x,y) \, dx dy, \qquad i=1,2,
$$
the  corresponding Fourier extension operators.  Then, if the constants $c_0$ and $\de\ll1$ in \eqref{U1}, \eqref{U2}  are  sufficiently small,
\begin{align}\label{bilinest}
	 \|\ext_{U_1}(f_1),\ext_{U_2}(f_2)\|_p \leq C_p \, \delta^{\frac{7}{2}-\frac{6}{p}}
\|f_1\|_2\|f_2\|_2
 \end{align}
 for every $f_1\in L^2(U_1)$ and every  $f_2\in L^2(U_2),$
where the constant $C_p$  will  only depend on $p$ and the constants $C_l$ in \eqref{cuberrorF}.
\end{thm}
Since according to Lemma \ref{sizeofdeltas}  the transversalities  $\tau_{z_1}(z_1,z_2)$ and $\tau_{z_2}(z_1,z_2)$ are of quite different sizes when $\de\ll 1,$ as in \cite{bmv19}
we first apply  a scaling transformation in $x$ by  a factor $\de$ in order to adjust these transversalities to the same level:

We introduce new coordinates $(\bar x,\bar y)$ be writing $x=\de \bar x, y=\bar y,$ and then re-scale the phase function $\phi$ by putting
$$\phi^s(\bar x,\bar y):=\frac{1}{\de}\phi(\de \bar x, \bar y)=\bar x \bar y+\frac {F(\bar y)}{\de}.$$
 Denote by $U_i^s$ the corresponding re-scaled domains, i.e.,
 \begin{eqnarray*}
U^s_1&=&\{(\bar x_1,\bar y_1): 0\le \bar x_1< c_0^2 ,\ 0\le  \bar y_1< c_0\delta   \},\\
U^s_2&=&\{(\bar x_2,\bar y_2): 0\le  \bar x_2+\frac {F'(\bar y_2)}{\de}-\bar a< c_0^2, \ 0\le  \bar y_2-\bar b< c_0\},
\end{eqnarray*}
where $c_0$ is small and $|\bar a|=|a/\de|\sim 1$ and $\bar b=b\sim 1.$
By $S_i^s, i=1,2,$ we denote the corresponding   scaled surface patches
$$
S_i^s:=\{(\bar x,\bar y,\phi^s(\bar x,\bar y)): (\bar x,\bar y)\in U_i^s\}.
$$

Observe that
\begin{align*}
	\nabla\phi^s(\bar x,\bar y)=(\bar y,\bar x+F'(\bar y)/\de),
\end{align*}
and
\begin{equation}\label{Hphi}
H\phi^s(\bar x,\bar y)=
	\left(\begin{array}{cc}
    0  & 1  \\
    1 & F''(\bar y)/\de
\end{array}\right),
\end{equation}
so that in particular
\begin{align}\label{bddgrad}
	|\nabla\phi^s(\bar z)|\lesssim 1
\end{align}
for all $\bar  z\in U^s_1\cup U^s_2$.

Assume next that  $\bar  z_1\in U^s_1$ and $\bar  z_2\in U^s_2.$  Since $|\bar y_1|\le c_0\de,|\bar y_2|\sim 1,$  we see that
 \begin{equation}\label{Fprimes}
 \begin{cases}  |\frac{F'(\bar y_1)}\de |\sim\frac{|F'''(\eta_1)\bar y_1^2|}\de \lesssim C_3c_0^2\frac{\de^2}\de=c_0^2 C_3\de, \quad
\frac{|F''(\bar y_1)|}\de \sim \frac{|F'''(\tilde\eta_1)\bar y_1|}\de \lesssim c_0C_3,\\
 |\frac{F'(\bar y_2)}\de |\sim\frac{|F'''(\eta_2)\bar y_2^2|}\de \sim\frac{ C_3}\de, \quad
\frac{|F''(\bar y_2)|}\de \sim \frac{|F'''(\tilde\eta_2)\bar y_2|}\de \sim \frac{C_3}\de
\end{cases}
\end{equation}
(for suitable choices of intermediate points $\eta_i,\tilde \eta_i$).  Moreover, we then also see that
\begin{equation}\label{graddiffs}
\nabla\phi^s(\bar z_2)-\nabla\phi^s(\bar z_1) =\big(\bar y_2-\bar y_1,\bar x_2+\tfrac {F'(\bar y_2)}{\de}-(\bar x_1+\tfrac {F'(\bar y_1)}{\de})\big)
	=(\bar b, \bar a)+\Landau(c_0).
\end{equation}
	
Following further on the proof of  Lemma 2.3 in \cite{bmv17}, assume that we translate the two patches of surface  $S_1^s$ and $S_2^s$ in such a way  that the two points $\bar z_1$ and $\bar z_2$  coincide after translation, and  assume that the vector
 $\omega=(\omega_1,\omega_2)$ is tangent to the corresponding   intersection curve $\gamma(t)$  at this point. Then \eqref{graddiffs} shows that we may assume without loss of generality that
\begin{align}\label{tangent}
	\omega=(-\bar a,\bar b)+\Landau(c_0).
\end{align}
In combination with \eqref{Fprimes} this implies that
\begin{align*}
	H\phi^s(\bar z_i)\cdot\trans\omega = \left(\begin{array}{cc}
    0  & 1  \\
    1 & F''(\bar y_i)/\de
\end{array}\right)
\left(
\begin{array}{cc}
  -\bar a+ \Landau(c_0)  \\
\bar b+ \Landau(c_0)  \\
\end{array}
\right).
\end{align*}
Thus, if $i=1,$ then by \eqref{Fprimes},
\begin{equation}\label{H1}
H\phi^s(\bar z_1)\cdot\trans\omega =
\left(
\begin{array}{cc}
  \bar b+ \Landau(c_0) \\
-\bar a  + \Landau(c_0) \\
\end{array}
\right)
\quad \text{ and   }\quad  |H\phi^s(\bar z_1)\cdot\trans\omega| \sim 1,
\end{equation}
and if $i=2,$ then
\begin{equation}\label{H2}
H\phi^s(\bar z_2)\cdot\trans\omega =
\left(
\begin{array}{cc}
  \bar b+ \Landau(c_0) \\
-\bar a  + \bar b F''(\bar y_2)/\de+\Landau(c_0)/\de \\
\end{array}
\right)
\quad \text{ and   }\quad  |H\phi^s(\bar z_1)\cdot\trans\omega| \sim 1/\de,
\end{equation}
if $\de\ll1$ is sufficiently small.
\smallskip

Following \cite{bmv17}, the  refined transversalities that we need to control in order to apply the bilinear method in the next section are given by
\begin{equation}\label{transnew}
\Big|TV^s_i(\bar z_1,\bar z_2)\Big|:=\Big|
\frac{\det(\trans (\nabla\phi^s(\bar
z_1)-\nabla\phi^s( \bar z_2)),  H\phi^s( \bar z_i)\cdot\trans\omega)}
{\sqrt{1+|\nabla\phi^s( \bar z_1)|^2}\sqrt{1+|\nabla\phi^s(\bar z_2)|^2}\, |H\phi^s(\bar
z_i)\cdot\trans\omega|}\Big|, \qquad i=1,2.
\end{equation}
	
	But, if $i=1,$ then by \eqref{graddiffs}, \eqref{H2}, \eqref{bddgrad}  and \eqref{Fprimes} we see that
$$
	|\det(\trans (\nabla\phi^s(\bar
z_1)-\nabla\phi^s( \bar z_2)),  H\phi^s( \bar z_1)\cdot\trans\omega)|=\Big|\det\left(
\begin{array}{ccc}
    \bar b+\Landau(c_0)  &  \bar b+\Landau(c_0)  \\
  \bar a+\Landau(c_0)    &  -\bar a+\Landau(c_0) \\
\end{array}
\right)
\Big|\sim 1,
$$
hence $\Big|TV^s_1(\bar z_1,\bar z_2)\Big|\sim1$.

 And, if $i=2,$ then by \eqref{graddiffs}, \eqref{H1}, \eqref{bddgrad}  and \eqref{Fprimes} we have
$$
	|\det(\trans (\nabla\phi^s(\bar
z_1)-\nabla\phi^s( \bar z_2)),  H\phi^s( \bar z_2)\cdot\trans\omega)|=\Big|\det\left(
\begin{array}{ccc}
    \bar b+\Landau(c_0)  &  \bar b+\Landau(c_0)  \\
  \bar a+\Landau(c_0)    &  -\bar a+\bar b F''(\bar y_2)/\de +\Landau(c_0)/\de \\
\end{array}
\right)
\Big|,
$$
hence also $\Big|TV^s_2(\bar z_1,\bar z_2)\Big|\sim (1/\de)/(1/\de)\sim 1,$  provided $\de$ and $c_0$ are
sufficiently small.
\medskip

We have thus proved the following lemma:

\begin{lemma}\label{transvscaled}
The transversalities for the scaled patches of surface  $S_i^s, \,i=1,2,$ satisfy
$$
\Big|TV^s_i(\bar z_1,\bar z_2)\Big|\sim 1, \quad i=1,2.
$$
\end{lemma}

In the next section, we shall establish the crucial bilinear Fourier extension estimates for the pair of scaled patches of surface $S^s_1$ and $S^s_2.$


\setcounter{equation}{0}
\section{The bilinear method with slowly decaying wave packets for the re-scaled prototypical case }\label{bilinprot}

First, recall  from the previous Subsection \ref{proto} that we had passed from our original coordinates $(x,y)$ to the coordinates $(\bar x, \bar y)$ by means of the scaling transformation $(x,y)=A(\bar x, \bar y)=(\delta\bar x,\bar y),$ and had put 
$\phi^s(\bar z):=\phi(A\bar z)/{\mathfrak a},$ with  $\mathfrak a:=\det A=\delta$.

 By means of simple  scaling argument (compare Subsection 3.1 of  \cite{bmv17}), estimate \eqref{bilinest}  is equivalent the following  bilinear Fourier extension estimate for the scaled patches of surface $S_i^s$ which where defined as the  graphs of $\phi^s$ over the sets $U^s_i, i=1,2:$ 
\begin{align}\label{bilinestscaled}
	 \|\ext_{U^s_1}(f_1)\,\ext_{U^s_2}(f_2)\|_p \leq C_p \,
\delta^{\frac{5}{2}-\frac{4}{p}} \|f_1\|_2\|f_2\|_2,
 \end{align}
 for every $f_1\in L^2(U^s_1)$ and every  $f_2\in L^2(U^s_2).$ 
  \medskip
  
  In order to defray our notation, let us in the sequel drop the superscripts $s,$  and write simply $U_i$ and $S_i$ in place of $U^s_i$ and $S^s_i,$ $\phi$ in place of $\phi^s,$  and  denote the coordinates $(\bar x, \bar y)$  again by $(x,y),$ etc.. I.e., we assume that 
 $$\phi( x,y)=xy+\frac {F(y)}{\de},$$
 where $F$ is a perturbation function of coarse cubic type in the sense of \eqref{cuberrorF}, i.e., 
 \begin{eqnarray*}
F(0)=F'(0)=F''(0)=0 \quad \text{and}\quad  {C_3}/4\le F'''(y)\le C_3 \quad \text{for all}\  y,
\end{eqnarray*}
and that 
 \begin{eqnarray}\label{Us}
 \begin{split}
U_1&=\{( x_1, y_1): 0\le  x_1< c_0^2 ,\ 0\le  y_1< c_0\delta   \},\\
U_2&=\{( x_2,y_2): 0\le   x_2+\frac {F'(y_2)}{\de}- a< c_0^2, \ 0\le y_2- b< c_0\},
\end{split}
\end{eqnarray}
 where  $c_0$ is assumed to be sufficiently  small and $| a|\sim 1$ and $b\sim 1.$  The corresponding transversalities will then be denoted by $TV_i( z_1, z_2),$ where $z_i=(x_i,y_i),\ i=1,2.$ 
 \smallskip

 Note that for $z_1 \in U_1$ and  $z_2\in U_2$, and $k=0,1,2,3$, we have $|F^{(k)}(y_1)|\sim|y_1|^{3-k}\lesssim \delta^{3-k}$, whereas $|F''(y_2)|\sim 1.$   
 \smallskip
 
 In view of \eqref{Hphi}, \eqref{bddgrad} and Lemma \ref{transvscaled}, the following properties of $\nabla \phi,$  the   operator norm of the Hessian matrix $H\phi,$  and of the transversalities  of $\phi$ are immediate or easy to verify:  
 \smallskip
 
For all $z_1\in U_1$, $z_2\in U_2$,  we have 
\begin{eqnarray}\label{grad1}
|\nabla\phi(z_i)|&\lesssim& 1\qquad i=1,2;\\
\Big|TV_i(z_1,z_2)\Big|&\sim& 1, \qquad i=1,2; \label{tv3012}	\\
\|H(\phi)(z_1)\| \sim  1, && \ \|H(\phi)(z_2)\|\sim 1/\delta; \label{cur}\\
|\partial_1^j\partial_2^k\phi(z_1)|\lesssim\delta^{2-j-k}, && \ |\partial_1^j\partial_2^k\phi(z_2)|\lesssim\delta^{-1}, \qquad \text{ if }  k\leq3, j+k\geq2.\label{deriv} 
\end{eqnarray}
Note that the last estimates can in some cases be improved, but they will be sufficient  for our purposes. With this modified notation,  our major goal  in this  section will  be to prove the following

\begin{thm}\label{bilinestn}
Under the preceding assumptions on $\phi$ and $U_1,U_2,$ the following bilinear estimate holds true  uniformly for every $f_1\in L^2(U_1)$ and every  $f_2\in L^2(U_2):$ 
$$
 \|\ext_{U_1}(f_1)\,\ext_{U_2}(f_2)\|_p \leq C_p \,
\delta^{\frac{5}{2}-\frac{4}{p}} \|f_1\|_2\|f_2\|_2.
$$
\end{thm}

Our  basic approach to this theorem  will follow  the lines in our preceding article \cite{bmv16}, which in turn is following the structure in Lee's article \cite{lee05}. We shall explain in more detail those parts where new ideas  are  needed,  and shall be brief about those arguments  which are essentially the same  as  in the classical bilinear method. For further details on the latter parts, we refer the reader to \cite{lee05}, \cite{v05}, and also \cite{T2}.

\subsection{Slowly decaying wave packets}\label{slowpack}
As a first step, we need to introduce a modified version of the so-called wave packet decomposition. Wave packet decompositions have become  by now a standard tool  for the study of problems in harmonic analysis related to geometric properties of submanifolds.  In those applications, hitherto one always had a  good control over all derivatives (up to a fixed, but arbitrary order) of the phase functions  arising in the given contexts.

In our present setting, however, we have to deal with phase functions $\phi(x,y)=xy+F(y)/\delta$, where $F$ only  satisfies the conditions in \eqref{cuberrorF}. In particular, we cannot assume any control over derivatives of $F$ of order 4 or higher! Consequently,  our ``wave packets" will lack a rapid decay away from the axes of  the tubes arising in the wave packet decomposition. Let us  briefly discuss how much we can salvage from the usual wave packet decomposition, and where we need modifications.

\smallskip

For a point $(\xi,\tau)\in\R^2\times\R$, we will write $\xi=(\xi^1,\xi^2)$, to distinguish from subscripts referring to $S_1$ and $S_2$.
\medskip 

In the next lemma, we  shall derive a slowly decaying wave packet decomposition for more general  classes of phase functions $\phi$  and corresponding surfaces given as the graph of $\phi$ over an open and bounded subset $U\subset\R^2,$  under suitable conditions on derivatives  of $\phi.$  By
$$
\ext_{U} f(\xi,\tau)=\int_{U} f(x,y) e^{-i(\xi^1 x+\xi^2 y+\tau\phi(x,y))} \alpha(x,y) \, dx dy,$$
 we denote again the corresponding Fourier extension operator, where $\alpha \in C_0^\infty$ is a suitable amplitude.
 \smallskip
 
In Subsection \ref{mainproof}, we shall later apply this lemma directly to $S_1,$ with $\kappa:=1,$ which is possible in view of \eqref{grad1} -- \eqref{deriv}.   Note here that for $R>1/\delta$ the set $ U'_1$  has essentially the same  dimension and shape as $U_1,$ so that  \eqref{grad1}--\eqref{deriv} can be assumed to hold true even on the larger  set $U'_1.$ The lemma could as well be applied to $S_2$, with $\kappa:= 1/\delta,$  since again  $ U'_2$  has essentially the same  dimension and shape as $U_2.$ However,   as for $S_2, $  we shall in Subsection \ref{mainproof} first re-parametrize $S_2$  in order to match our wave packet decomposition with  some problems arising in the induction on scales argument,  and then apply the lemma in the new coordinates.

\begin{lemma}\label{wave packets} 
Let $\delta>0$ and  $\kappa\geq1.$  Let $U\subset\R^2$ be an open and bounded subset and  let $ U':=U+\Landau(\delta)$ be  suitable thickening of the set $U$ of order $\delta. $ Assume further that $\phi$ a  smooth phase function  defined on $U'$ such that 
\begin{equation}\label{wavecond}
\|H\phi(z)\|\sim\kappa\ge 1\text {  and  } |\partial_1^j\partial_2^k\phi(z)|\lesssim(\delta\kappa)^{3-k-j}\delta^{-1}
\end{equation}
for every  $z\in U'$ and $k\leq3$, $j+k\geq2$. Let us also assume that $\partial_1^2\phi\equiv0.$ 

Then for every $R\ge 1/\delta$ the function $\ext_{U} f$ can be decomposed into slowly decaying  wave packets adapted to $\phi$ with  central tubes of radius
$R$ and length $R^2/\kappa$.
More precisely,  consider the   index sets  $\Y:=R\Z^2$ and  $\V:=R^{-1}\Z^2 \cap  U'$,  and  define for $w=(\eta,v)\in \Y\times \V=:\W$ the tube
\begin{align}\label{wavepack}
	T_w= \{(\xi,\tau)\in\R^2\times\R: |\xi-\eta+\tau\nabla\phi(v)|\leq R,\ |\tau|\leq R^2/\kappa \}.
\end{align}
Then there exist functions (``slowly decaying wave packets'') $p_w$ and coefficients $c_w\in\C$, $w\in \W$, such that  $\ext_{U}f$ can be decomposed into

$$
\ext_{U}f(\xi,\tau)=\sum\limits_{w\in\W} c_w p_w(\xi,\tau)
$$
for every $\tau\in\R$ with  $|\tau|\leq{R^2}/{\kappa},$ $\xi=(\xi^1,\xi^2)\in\R^2$, in such a way that  the following hold true:
\begin{enumerate}
\renewcommand{\labelenumi}{(P\arabic{enumi})} 
	\item $p_w=\ext_{U}(\F^{-1} (p_w(\cdot,0))).$
	\item $\supp \F p_w \subset B((v,\phi(v)),2/R).$
	\item $p_w$ decays away from $T_w$  as follows:
					$$|p_w(\xi,\tau)|\leq C R^{-1} \left(1+\frac{|\xi^1-\eta^1+\tau\partial_1\phi(v)|}{R}\right)^{-2}
					\left(1+\frac{|\xi^2-\eta^2+\tau\partial_2\phi(v)|}{R}\right)^{-2}.$$
			In particular,  $\|p_w(\cdot,\tau)\|_2 \lesssim 1$.
	\item For all subsets $W\subset \W$, we have
					$ \|\sum\limits_{w\in W} p_w(\cdot,\tau)\|_2\lesssim |W|^\frac{1}{2}$.
	\item $\|(c_w)_w\|_{\ell^2} \lesssim \|f\|_{L^2}$.					
\end{enumerate}
\renewcommand{\labelenumi}{(\roman{enumi})} 
\end{lemma}

\begin{remarks} 
(a) Observe that property (P3) implies the weaker estimate
\begin{align}\label{disttw}
	|p_{w}(\xi,\tau)|\leq C R^{-1} \left(1+\frac{\dist((\xi,\tau),T_{w})}{R}\right)^{-2},
\end{align}
which means that the mass of $p_w$ is in a certain weak sense mostly localized to $T_w$ (however, the decay outside $T_w$ is now  much slower than classically).
This estimate will obviously hold true later in Subsection \ref{mainproof} for the wave packets associated to $\ext_{U_1} f_1,$ but as well for those associated to  $\ext_{U_2} f_2,$ and in the last parts of the proof where interactions between wave packets of type 1 with wave packets of type 2 will become relevant we can just rely on these weaker estimates \eqref{disttw} in place of the more refined ones given by $(P3).$ 

(b) The condition $\partial_1^2\phi\equiv0$ will be satisfied in our applications and simplifies the integrations by parts argument in the proof a lot, but is quite  surely not  necessary. We shall, however, not dwell on this here.
\end{remarks}

\begin{proof} We shall closely follow  our preceding paper \cite{bmv16} and  mainly focus on those  parts of the proof which will require new  arguments. 
\smallskip

The initial construction remains the usual one, as in (for instance) \cite{lee05}:
 Let $\psi,\hat\eta\in C_0^\infty(B(0,1))$ be chosen in a  such a way that for
$\psi_v(z):=\psi(R(z-v))$ and  $\chi_\eta(\xi):=\chi(\frac{\xi-\eta}{R}),$
we have $\sum\limits_{v\in\V}\psi_v=1$ on $U$ and $\sum\limits_{\eta\in\Y} \chi_\eta=1$.
We also chose a slightly bigger function  $\tilde\psi\in C_0^\infty(B(0,3))$ such that $\tilde\psi=1$ on $B(0,2)\supset\supp\psi+\supp\hat\eta,$ and put 
$\tilde\psi_v(z):=\tilde\psi(R(z-v)),$ where we again assume that $z=(x,y).$ 
Then the functions
$$
F_{(\eta,v)}:=\F^{-1}(\widehat{\psi_vf}\chi_\eta)=(\psi_v f)\ast \check\chi_\eta, \qquad \eta\in\Y, v\in\V,
$$
 are  essentially well-localized in both position and momentum/frequency  space and supported in balls of radius $\sim 1/R.$  Note also that $f$ supported in $U$  we have $\ext_{U} f=\ext f,$  Define then 
$$
q_w:=\ext(F_w), \qquad w=(\eta,v)\in\W;
$$
up to a certain factor $c_w,$ which will be determined later, these are already the announced wave packets, i.e., $q_w=c_w p_w$.
\smallskip

We may also assume for simplicity that the amplitude in the definition of the extension operator $\ext$ is identically 1 on the sets $U_i.$ Since $f=\sum\limits_{w\in W} F_w $, we then have the decomposition  $\ext_{U}f=\sum\limits_{w\in W} q_w.$ 
Let us concentrate on property (P3) - the other properties are then rather easy to establish. It is easy to see that we have for every $w\in\W$
\begin{align*}
	q_w(\xi,\tau) 
	=& (2\pi)^{-2} R^{-2}\int K(\xi-\zeta,\tau) \widehat {F_w}(\zeta)\, d\zeta,					
\end{align*}
with the kernel
\begin{align} \nonumber
	K(\xi,\tau):=\int e^{i\big(\xi(\frac{z}{R}+v)+\tau\phi(\frac{z}{R}+v)\big)}\tilde\psi(z) \alpha(\frac{z}{R}+v) \, dz.
\end{align}

To simplify the proof a bit, let us pretend in the sequel that $\alpha(\frac{z}{R}+v)\equiv 1$  on the support of $\tilde \psi;$  since all derivatives of $\alpha(\frac{z}{R}+v)$ are uniformly bounded on the support of $\tilde \psi,$  we shall see that the presence of the factor $\alpha(\frac{z}{R}+v)$ has indeed no effect on the subsequent integration by parts arguments.

\smallskip
We shall show  that
\begin{align}\label{oszint2}
	|K(\xi,\tau)| \lesssim \left(1+\frac{|\xi^1+\tau\partial_1\phi(v)|}{R}\right)^{-2}
					\left(1+\frac{|\xi^2+\tau\partial_2\phi(v)|}{R}\right)^{-2}.
\end{align}
 To this end, we may  clearly assume that
\begin{align}\label{eins}
	|\xi^1+\tau\partial_1\phi(v)|\gg R,
\end{align}
or
\begin{align}\label{zwei}
	|\xi^2+\tau\partial_2\phi(v)|\gg R.
\end{align}
 Let us start with the case where both conditions hold true.  We first perform integrations  by parts with respect to the first component $x$ of  $z=(x,y) $ in the oscillatory integral
\begin{align}\nonumber
	K := \int e^{i\Phi(z)}\tilde\psi(z)\,dz,
\end{align}
with phase
\begin{align}\nonumber
	\Phi(z) := \xi(\frac{z}{R}+v)+\tau\phi(\frac{z}{R}+v).
\end{align}
Recalling that  $\partial_1^2\phi\equiv0,$ hence
$\partial_1^2\Phi(z)\equiv0,$ we see that  
\begin{align*}
	K=\int ie^{i\Phi(z)} 
	\partial_1\frac{\tilde\psi(z)}{\partial_1\Phi(z)} \,dz	
	= i\int e^{i\Phi(z)} 
	\frac{\partial_1\tilde\psi(z)}{\partial_1\Phi(z)} \,dz.
\end{align*}
Repeating this argument $N$ times, we obtain
\begin{align}\label{reyes}
	K= i^N\int e^{i\Phi(z)} 
	\frac{\partial_1^N\tilde\psi(z)}{[\partial_1\Phi(z)]^N} \,dz.
\end{align}
To shorten the notation in the subsequent  computation, let  us put $A(z):=\dfrac{i^N\partial_1^N\tilde\psi(z)}{[\partial_1\Phi(z)]^N}$.
Integrating by parts twice now with respect to the second component $y$ of  $z=(x,y),$ we obtain
\begin{align*}
	K =& \int ie^{i\Phi(z)} 
	\partial_2\frac{A(z)}{\partial_2\Phi(z)} \,dz	\\
	=& -\int ie^{i\Phi(z)} 
	\left(\frac{\partial_2^2\Phi(z)}{(\partial_2\Phi(z))^2} A(z)- \frac{\partial_2 A(z)}{\partial_2\Phi(z)}\right) \,dz	\\
	=& \int e^{i\Phi(z)} \partial_2
	\left(\frac{\partial_2^2\Phi(z)}{(\partial_2\Phi(z))^3} A(z)- \frac{\partial_2A(z)}{(\partial_2\Phi(z))^2}\right) \,dz	\\
	=& \int e^{i\Phi(z)} 
	\left(\frac{\partial_2^3\Phi(z)}{(\partial_2\Phi(z))^3} A(z)-3\frac{(\partial_2^2\Phi(z))^2}{(\partial_2\Phi(z))^4}A(z)+3
	\frac{\partial_2^2\Phi(z)}{(\partial_2\Phi(z))^3}\partial_2A(z)- \frac{\partial_2^2A(z)}{(\partial_2\Phi(z))^2}\right) \,dz	\\
	=& \int e^{i\Phi(z)} 
	\left(\frac{\partial_2^3\Phi(z)}{(\partial_2\Phi(z))^3} A(z)-3\frac{(\partial_2^2\Phi(z))^2}{(\partial_2\Phi(z))^4}A(z)+3
	\frac{\partial_2^2\Phi(z)}{(\partial_2\Phi(z))^3}\partial_2A(z)\right) \,dz\\
	&\hskip2cm - \int e^{i\Phi(z)}\frac{\partial_2^2A(z)}{(\partial_2\Phi(z))^2}\,dz.
\end{align*}
For the first term, we stop integrating by parts in $y$, since we have no control of derivatives of $\Phi$ of order higher than 3. However, the second term can be integrated by parts once more.

Notice also that for all $j,k\in \N$ such that $ k\le 3$  and $j+k\ge 2$ we have 
$$
|\partial_1^j\partial_2^k \Phi(z)|\lesssim \frac{|\tau|(\delta\kappa)^{3-j-k}}{\de R^{j+k}}
\leq (\delta\kappa R)^{2-j-k} \leq 1, 
$$
since we assume that $|\tau|\kappa\leq R^2$, $\de\kappa R\geq\de R\ge 1.$ \color{black} Therefore, we see that for every $N\in \N$
\begin{align}
	|K| \lesssim C_N \int_{\supp\tilde\psi} |\partial_1\Phi(z)|^{-N}|\, \partial_2\Phi(z)|^{-3}dz. 
\end{align}

Observe next that for  $|\tau|\leq{R^2}/{\kappa}$ and $z\in\supp\tilde\psi,$ in view of \eqref{wavecond} we have for $j=1,2$ 
$$
|\tau\partial_2\phi(z/R+v)-\tau\partial_j\phi(v)|\lesssim |\tau |(|\pa_1\partial_j \phi|+|\pa_2\pa_j \phi|)\frac {|z|}R\lesssim \frac {R^2}\kappa \cdot \kappa\cdot \frac 1 R \lesssim R.
$$
In view of \eqref{eins},\eqref{zwei}, these estimates imply that for $j=1,2,$ 
$$
|\partial_j\Phi(z)|=|\xi^j/R+\tau\partial_j\phi(z/R+v)/R|\gtrsim |\xi^j+\tau\partial_j\phi(v)|/R,
$$
so we obtain \eqref{oszint2} for this case.

If only one of the conditions \eqref{eins} or \eqref{zwei} holds, we only integrate by parts in one of the variables and arrive again at  \eqref{oszint2}.

Following the proof in \cite{lee05}, we conclude that
\begin{eqnarray*}
	|q_w(\xi,\tau )| &\lesssim&  R^{-2}\int\left| K(\xi-\zeta-\eta,\tau)	\widehat {F_w}(\zeta+\eta) \right|d \zeta		\\
	&=& R^{-2}\int\left|  K(\xi-\zeta-\eta,\tau)\chi\left(\frac{\zeta}{R}\right)\widehat{\psi_v f}(\zeta+\eta) \right|d \zeta	\\
	&\lesssim&  \mathop{M} (\widehat{\psi_v f})(\eta)
	 \left(1+\frac{|\xi^1-\eta^1+\tau\partial_1\phi(v)|}{R}\right)^{-2} 
					\left(1+\frac{|\xi^2-\eta^2+\tau\partial_2\phi(v)|}{R}\right)^{-2},
\end{eqnarray*}
where $M$ denotes the Hardy-Littlewood maximal operator. Thus, by  choosing  $c_w=c_{(\eta,v)}:=R M(\widehat{\psi_vf})(\eta),$ we obtain (P3).
\medskip

Properties (P1) and (P2) follow from the definition of the wave packets as usual.

To prove  (P4), as a first step note that  it is easily seen  by Plancherel's theorem and  (P2) that 
$$ \|\sum\limits_{w\in W} p_w(\cdot,\tau)\|_2\lesssim
 \left(\sum_v\|\sum\limits_{\eta:(\eta,v)\in W} p_{(\eta,v)}(\cdot,\tau)\|_2^2\right)^{1/2} .
 $$
As we lack fast decay in (P3), we need to be a little more careful in establishing the orthogonality in the $y$-parameter in (P4). We have to show that for any $Y\subset\Y$
\begin{align*}
	\|\sum\limits_{\eta\in Y} p_{(\eta,v)}(\cdot,\tau)\|_2 \lesssim |Y|^{1/2}.
\end{align*}
By a linear transformation, we may assume that $v=0$. Then
\begin{eqnarray*}
\|\sum\limits_{\eta\in Y} p_{(\eta,v)}(\cdot,\tau)\|_2^2 
	&=& \sum_{\eta,\eta'\in Y} \int p_{\eta,v}(\xi,\tau)\,\bar p_{\eta',v}(\xi,\tau) d\xi\\
	&\leq& C R^{-2} \sum_{\eta,\eta'\in Y}\prod_{i=1}^2 \int \left(1+\frac{|\xi^i-\eta^i|}{R}\right)^{-2}
					\left(1+\frac{|\xi^i-{\eta'}^i|}{R}\right)^{-2}d\xi^i	\\
& =& C \sum_{\underset{Rk,Rk'\in Y}{k,k'\in\Z^2:}}\prod_{i=1}^2 
 							\int \left(1+|\xi^i-k^i|\right)^{-2}	(1+|\xi^i-{k'}^i|)^{-2} d\xi^i	\\
 &\leq& C \sum_{\underset{Rk\in Y}{k\in\Z^2:}}\prod_{i=1}^2 
 						\int \left(1+|\xi^i-k^i|\right)^{-2}	\sum_{k_i'\in\Z}(1+|\xi^i-{k'}^i|)^{-2} d\xi^i	\\
& \lesssim& |Y|.					
\end{eqnarray*}

For (P5), we refer to \cite{lee05}.
\end{proof}

\subsection{Proof of Theorem \ref{bilinestn} - Setup of the argument}\label{mainproof}

Due to a by now standard argument, it is sufficient to prove the following local estimate: For any $p>5/3$ and $\epsilon>0$, there exist a constant $C_{p,\epsilon}$ such that

\begin{align}
	 \|\ext_{U_1}(f_1),\ext_{ U_2}(f_2)\|_{L^p(Q(R'))} \leq C_{p,\epsilon} R'^\epsilon \, \delta^{\frac{5}{2}-\frac{4}{p}}
\|f_1\|_2\|f_2\|_2
 \end{align}
for all $f_1\in L^2(U_1)$, $f_2\in L^2( U_2)$ and $R'\gg1$, where $Q(R')$ is a cuboid that we  will choose below.

As already mentioned in Subsection \ref{slowpack}, we can apply the wave packet decomposition from Lemma \ref{wave packets}  directly to  $\ext_{U_1}f_1,$ with $\kappa=1.$
The wave packets associated  to $\ext_{U_1}f_1$  and the patch of hypersurface  $S_1$ are mainly  concentrated on tubes $T_{w_1}$,$w_1=(\eta_1,v_1)\in\W_1,$  which are ``horizontal''  translates  of tubes of the form
$$
 T_{(v_1)}:=\{(\xi,\tau):|\xi+\tau\nabla\phi(v_1)|\leq R, |\tau|\leq R^2\},
$$
but, in contrast to the classical situation, they do not decay rapidly away from those tubes. 
 Notice the
standard fact that $T_{(v_1)}$ is a tube of dimension $R\times R\times  R^2$ whose long axis is
pointing in the direction  of the normal vector $N(v_1):=(\nabla\phi(v_1),-1)$ to $S_1$ at
the point $(v_1,\phi(v_1)),$ $v_1\in U_1$.

\smallskip

Note next that  the projections of tubes $T_{w_1}$ to $\xi$-space are contained in a $\delta R^2\times R^2$-rectangle. Therefore we define $Q(R')$ to be the  $\delta R^2\times R^2\times R^2=\frac{R'^2}{\delta}\times\frac{R'^2}{\delta^2}\times\frac{R'^2}{\delta^2}$ cuboid centered at the origin. Then for any $(\xi,\tau)\in Q(R')$, we have $|\tau|\leq R^2$, so we can decompose 
 \begin{equation}\label{wavep1}
\ext_{U_1} f_1(\xi,\tau)=\sum_{w_1\in\W_1} c_{w_1}p_{w_1}(\xi,\tau).
\end{equation}

\medskip
As for the wave packets  associated  to $\ext_{U_2}f_2$  and the patch of hypersurface  $S_2,$  these should then  mainly  be concentrated on shorter tubes $T_{w_2}$,   of length $\de R^2$ in order to fit well into the cuboid $Q(R').$ Moreover, we would like to translate them not horizontally, but in directions of the $\xi^2,\tau$ coordinates so that they can fill up $Q(R').$
\smallskip

For $\ext_{ U_2} f_2$, we therefore need to switch coordinates first: Since  $y_2\sim 1$ on $U_2,$  we may solve the equation $u=\phi(x,y)$ for $x,$ i.e., 
$x=\frac {u-F(y)/\de}y,$ 
and can accordingly  re-parametrize $S_2=\{(x,y,\phi(x,y)):(x,y)\in U_2\}$ in the form\
\begin{eqnarray*}
S_2&=&\{(\tilde\phi(y,u),y,u):(y,u)\in\tilde U_2\}, \text{    with }\\
\tilde\phi(y,u)&:=&\frac {u-F(y)/\de}y,\\ 
\tilde U_2&:=&\{( y,u): 0\le   u+\frac {F'(y)-\frac {F(y)}{y}}{\de}- a< c_0^2, \ 0\le y- b< c_0\}.
\end{eqnarray*}
Since $|y|\sim1$ for $(y,u)\in\tilde U_2$, it is easy to see that (compare \eqref{deriv})
\begin{eqnarray*}
|\partial_1^j\partial_2^k\tilde\phi(y,u)|\lesssim\delta^{-1} \qquad \text{and  } \|H(\tilde \phi)(y,u)\|\sim 1/\delta.
\end{eqnarray*}

This suggests to change variables to $\tilde\xi^1:=\tau$, $\tilde\tau:=\xi^1$ and $\tilde\xi^2:=\xi^2$,
so that the new $\tilde\tau$-coordinate points in the shorter $\xi^1$-direction of $Q(R')$. 
Applying Lemma \ref{wave packets} in the new coordinates with $\kappa:=1/\delta,$ we obtain a slowly decaying wave packet decomposition w.r. to the coordinates $(\tilde\tau,\tilde\xi^2,\tilde\xi^1).$ For  instance, (P3) becomes
\begin{align*}
|p_{w_2}(\xi,\tau)|\leq C R^{-1} \left(1+\frac{|\tilde\xi^1-\eta_2^1+\tilde\tau\partial_1\tilde\phi(\tilde v_2)|}{R}\right)^{-2}
					\left(1+\frac{|\tilde\xi^2-\eta_2^2+\tilde\tau\partial_2\tilde\phi(\tilde v_2)|}{R}\right)^{-2}.	
\end{align*}
The parameters $\tilde v_2$ correspond the parameters $v_2$ in the original coordinates through the relation 
$$
(v_2, \phi(v_2))=(\tilde\phi(\tilde v_2), \tilde v_2).
$$
In particular, the  wave packets  $p_{w_2}$ associated  to $\ext_{U_2}f_2$  and the patch of hypersurface  $S_2$ are mainly  concentrated on tubes $T_{w_2}$  whose axis   point in the direction of the normal to $S_2$ at the point $(v_2, \phi(v_2)) $ and are translates  of tubes of the form
$$
 T_{(v_2)}:=\{(\xi,\tau):|\xi+\tau\nabla\phi(v_2)|\leq R, |\tau|\leq \delta R^2\},
$$
w.r to the coordinates $(\xi^2,\tau).$ We parametrize these tubes again by a pair of parameters  $w_2=(\eta_2,v_2)\in\W_2,$ where $\eta_2$ represents the translation parameter.
 
 \smallskip
Then for any $(\xi,\tau)=(\tilde\tau,\tilde\xi^2,\tilde\xi^1)\in Q(R')$, we have $|\tilde\tau|\leq \delta R^2$, so we can decompose 
 \begin{equation}\label{wavep2}
 \ext_{ U_2} f_2(\xi,\tau)=\sum_{w_2\in\W_2} c_{w_2}p_{w_2}(\xi,\tau).
 \end{equation}

This kind of change of coordinates and the effect on the wave packet decomposition was introduced already in our previous paper \cite{bmv16}, therefore we will omit the details. 
\smallskip

Observe that the intersection of two transversal tubes $T_{w_1}$ and $T_{w_2}$ is essentially a cube $q$ of side length $R$. 

Another difference to the classical method occurs in the localisation to these cubes $q$. We want to choose functions $\chi_q$ which decay rapidly away from $q$ (by scaling a fixed function $\chi$ in the usual way), but so that  the Fourier transform $\hat\chi_q$ has compact support in a $1/R$ cube centered at the origin. We can do this in a manner so that still $\sum_q\chi_q=1$ on $Q(R')$. The advantage is that then $p_{w_1}\chi_q$ essentially behaves like a wave packet again. In particular, we still have a well localised Fourier support, compare property (P2) from Lemma \ref{wave packets}:
\begin{align}\label{pdos}
	 \supp \F(p_w\chi_q) \subset B((v,\phi(v)),3R^{-1}).
\end{align}
The same idea had already been applied in \cite{bmv16}.

\medskip

The above wave packet decompositions based on Lemma \ref{wave packets}  applies  for any  $R\ge 1/\delta$. 
Since we want to induct on the scale, we introduce $R':=\delta R$, so that we can induct over $R'\geq 1$ (compare again with \cite{bmv16}). More precisely:
 
\begin{defn} For any $\alpha>0$, we say that {E($\mathbf{\alpha}$)} holds true if
\begin{align}\label{locbilinest}
	 \|\ext_{U_1}(f_1),\ext_{ U_2}(f_2)\|_{L^p(Q(R'))} \leq C_{p,\alpha} R'^\alpha \, \delta^{\frac{5}{2}-\frac{4}{p}}
\|f_1\|_2\|f_2\|_2
 \end{align}
for all $R'\geq 1$ and $f_1\in L^2(U_1)$, $f_2\in L^2( U_2)$. 
\end{defn}
Let us fix a sufficiently  small  constant $\gamma>0$ throughout this section. In our induction on scales argument, at various places  harmless powers of the form $R'^{c\gamma}$ will arise, with possibly different values of the constants $c>0.$  With a slight abuse of notation, we shall nevertheless denote all of them by the same letter $c.$ 

  We  will then  show that  hypothesis $E(\alpha)$ essentially implies $E(\alpha(1-\gamma))$. In this argument, we will apply $E(\alpha)$ to $\tilde R':={R'}^{1-\gamma}\geq 1$ in place of $R'.$ Observe here  the inequality $R\geq1/\delta$ is not stable under taking powers of $R;$ this is why  we cannot induct on the parameter $R,$  but have to pass to the parameter $R'.$

For classical wave packet decompositions, the fast decay of the wave packets away from their ``central'' tubes   allows, for instance,  to treat the $L^\infty$-norm of a sum of wave packets all pointing in the same direction as if they had disjoint supports. Even though our rough wave packets are no longer  rapidly  decaying (we only have the weaker decay given by (P3)), the following lemma still holds true. 

\begin{lemma}\label{franz}~
Let $q$ be a cube as before, and let $d_1,d_2\ge 0$. 
\begin{enumerate}
\item Let $v_1\in\V_1$ and $Y_1\subset\Y_1$ such that for any $\eta_1\in Y_1$ we have $\dist(T_{(\eta_1,v_1)},q)\gtrsim d_1  R $. Then
$$\|\sum_{\eta_1\in Y_1} p_{(\eta_1,v_1)}\chi_q\|_{\infty} \lesssim (1+d_1)^{-1}R^{-1}.$$
\item Let $v_2\in\V_2$ and $Y_2\subset\Y_2$ such that for any $\eta_2\in Y_2$ we have $\dist(T_{(\eta_2,v_2)},q)\gtrsim d_2  R.$  Then
$$\|\sum_{\eta_2\in Y_2} p_{(\eta_2,v_2)}\chi_q\|_{\infty} \lesssim (1+d_2)^{-1}R^{-1}.$$

\item	Let $p_{w_1},p_{w_2}$ be transversal wave packets as above. Then
\begin{align}\label{transvinters}
	\int_{\R^3} |p_{w_1}p_{w_2}| \, d(\xi,\tau) \lesssim R.
\end{align}
\end{enumerate}
\end{lemma}

\begin{proof}
We prove (i) - the proof of (ii) is analogous.   Let us also assume for simplicity that $v_1=0;$  the general case can be treated in the same manner. 
\smallskip

First, assume that $ d_1\geq1$. Fix $(\xi,\tau)\in \R^3.$ If $\dist((\xi,\tau),q)\ll  d_1R$, then $\dist(T_{(\eta_1,v_1)},(\xi,\tau))\gtrsim  d_1  R ,$ and thus
\begin{eqnarray*}
	\sum_{\eta_1\in Y_1} |p_{(\eta_1,v_1)}(\xi,\tau)|
	&\lesssim& R^{-1}\sum_{\eta_1\in Y_1} \Big(1+\frac{|\xi^1-\eta^1_1|}{R}\Big)^{-2}\Big(1+\frac{|\xi^2-\eta^2_1|}{R}\Big)^{-2} \\
	&\lesssim& R^{-1}\sum_{\substack{k^1,k^2\in \Z:\\ k^1+k^2\geq d_1}} (1+|k^1|)^{-2}(1+|k^2|)^{-2}	\\
	&\lesssim& (1+ d_1)^{-1}R^{-1}.
\end{eqnarray*}

On the other hand,  if $\dist((\xi,\tau),q)\gtrsim d_1R$, then  we only obtain
$$
\sum_{\eta_1\in Y_1} |p_{(\eta_1,v_1)}(\xi,\tau)|  \lesssim R^{-1}\sum_{k^1,k^2\in \Z} (1+|k^1|)^{-2}(1+|k^2|)^{-2}\leq R^{-1},
$$
but we can use the rapid decay of $\chi_q$ to estimate
\begin{align*}
	 \chi_q(\xi,\tau) \lesssim ( d_1R)^{-N}\ll  d_1^{-N},
\end{align*}
which combined gives an even better estimate than required for (i). 
\smallskip

If $ d_1\leq 1$, the proof is even simpler.
\medskip

For (iii), we only use the coordinate-free decay estimate \eqref{disttw}, which is indeed sufficient for \eqref{transvinters}:
Due to transversality, by an affine linear change of variables, we may assume that $T_{w_1}$ is parallel to, say, the $\tau$-axis and $T_{w_2}$ is parallel to the $\xi^1$-axis, and both tubes go through the origin.
Then
\begin{eqnarray*}
\int_{\R^3} |p_{w_1}p_{w_2}| d(\xi,\tau)
&\lesssim& R^{-2} \int_{\R^3} \left(1+\frac{|(\xi^1,\xi^2)|}{R}\right)^{-2}
										\left(1+\frac{|(\xi^2,\tau)|}{R}\right)^{-2} d(\xi,\tau)	\\
&=& R \int_{\R^3} \left(1+|(\xi^1,\xi^2)|\right)^{-2}
										\left(1+|(\xi^2,\tau)|\right)^{-2} d(\xi,\tau)	\\
&\leq& R \int_{\R^3} \left(1+|\xi^1|\right)^{-4/3}\left(1+|\xi^2|\right)^{-2/3-2/3}
									\left(1+|\tau|\right)^{-4/3} d(\xi,\tau)	\\
	&\lesssim& R.		
\end{eqnarray*}
													
\end{proof}

\subsection{Dyadic decompositions and reduction}
Similar to the classical bilinear approach, we have to count how many tubes  (and also dilates of them) arising in the wave packet decompositions  \eqref{wavep1}
and   \eqref{wavep2} can   interact at a given  point.

To this end, for any $W_j\subset\W_j, j=1,1,$ and cube  $q\subset Q(R')$, we define  
$$W_j^1(q):=\{w_j\in W_j: R'^\gamma T_{w_j}\cap q\neq\emptyset\}. $$
Since our wave packets have only slow decay away from the tubes, we also want to have control over the tubes which lie at a certain distance to  $q.$  Equivalently, that means that a hollow tube of bigger diameter passes through $q$. For any  dyadic $\nu$  with  $1<\nu<\nu_{\rm max}\sim R'^{1-2\gamma}$, we therefore define the {\it hollow tube} of diameter $\nu  R R'^\gamma$ by 
$$T_{w_j}^\nu:=\big(\nu R'^\gamma T_{w_j}\big)\setminus\big(\frac{\nu}{2} R'^\gamma T_{w_j}\big),$$
where we specify that $rT_{w_j}$ means scaling $T_{w_j}$ around the central axis by the factor $r$. Then let
\begin{align}\label{wjnu}
	W_j^\nu(q):=\{w_j\in W_j: T_{w_j}^\nu\cap q\neq\emptyset\}.
\end{align}
This definition extends to the case $\nu=1$ if we set $T_{w_j}^1:=R'^\gamma T_{w_j}$. 
Finally, we will see that the weak decay of our wave packets does not require a further decomposition beyond $\nu= \nu_{\rm max}$, hence we set $T_{w_j}^{ \nu_{\rm max}}:=\R^3\setminus(\frac{\nu_{\rm max}}{2}R'^\gamma T_{w_j})$, and define again $W_j^{ \nu_{\rm max}}(q)$ by \eqref{wjnu}. In the following, let $\D$ denote the set of all dyadic numbers $\nu$ with $1\leq\nu\leq \nu_{\rm max}$.
\smallskip

Decomposing dyadically as usually in the bilinear argument, we may assume $|W_j^\nu(q)|$ to be essentially constant, i.e., we only need to consider cubes $q$ from the set
\begin{align}
	Q^\mu:=\{q:|W_j^\nu(q)|\sim \mu_{\nu,j}\ \forall j=1,2,\ \nu\in\D\}.
\end{align}
Here, $\mu=(\mu_{\nu,j})_{\nu,j}$, $j=1,2$, $\nu\in\D$, is a collection of dyadic numbers.
Further we define
\begin{align}
	Q^{\mu,\nu}(w_j):=\{q\in Q^\mu:T_{w_j}^\nu\cap q\neq\emptyset\},
\end{align}
and for any collection of dyadic numbers $\lambda_j=(\lambda_{\nu,j})_{\nu},$ $ j=1,2,\nu\in\D,$ 
\begin{align}
	W_j^{\lambda_j,\mu}:=\{w_j\in W_j: |Q^{\mu,\nu}(w_j)|\sim\lambda_{\nu,j}\ \forall\nu\in\D\}, \qquad j=1,2.
\end{align}
Finally let
\begin{align}
	W_j^{\lambda_j,\mu,\nu_j}(q):=W_j^{\lambda_j,\mu}\cap W_j^{\nu_j}(q).
\end{align}
\smallskip

Using these dyadic decompositions, we can make the following reduction:
\begin{lemma}\label{reduct}
Assume that for all $W_1\subset\W_1$, $W_2\subset\W_2,$ all dyadic 
$\lambda_1,\lambda_2,\mu$ and all dyadic $\nu_1,\nu_2=1,\ldots, \nu_{\rm max},$ the following estimate holds true:
\begin{align}\label{waveest}
	\|\sum_{q\in Q^\mu}\sum_{w_1\in W_1^{\lambda_1,\mu,\nu_1}(q)}\sum_{w_2\in W_2^{\lambda_2,\mu,\nu_2}(q)}  p_{w_1}p_{w_2}\chi_q  \|_{L^p(Q(R'))} \leq C_{p,\alpha} R'^\alpha \, \delta^{\frac{5}{2}-\frac{4}{p}}
	|W_1|^{1/2}|W_2|^{1/2}.
\end{align}
Then for any $\gamma>0$, $E(\alpha+\gamma)$ holds true.
\end{lemma}

\begin{remark} There is actually only a logarithmic loss in place of $R'^\gamma$ in Lemma \ref{reduct}, but we will not dwell on that.
\end{remark}

Therefore we  may and shall  assume from now on  that $W_1\subset\W_1$, $W_2\subset\W_2$, and  $\lambda_1,\lambda_2$ and $\mu$ are fixed.

\smallskip

Following the next step in the bilinear method,  we divide $Q(R')$ into $R'^{\mathcal{O}(\gamma)}$ cubes $b$ which are translates of ${R'}^{-4\gamma}Q(R')=Q(R'^{1-2\gamma})$. For a fixed tube $T_{w_j}$, we denote by $b(w_j)$  a cube $b$ for which the cardinality of $\{q\in Q^{\mu,1}(w_j):q\cap b\neq\emptyset\}$ is maximal over all $b.$ If there are several such $b$, we pick one of them.

\smallskip
We say that $ w_j$ and $b$ are related and write $w_j\sim b,$ if $b$ is a neighbor of $b(w_j)$. If we split the set of integration in \eqref{waveest} into
\begin{eqnarray*}
&&\hskip-0.5cm\|\sum_{q\in Q^\mu}\sum_{w_1\in W_1^{\lambda_1,\mu,\nu_1}(q)}\sum_{w_2\in W_2^{\lambda_2,\mu,\nu_2}(q)} p_{w_1}p_{w_2}\chi_q  \|^p_{L^p(Q(R))}	\\
	&=& \sum_b \|\sum_{q\in Q^\mu}\sum_{w_1\in W_1^{\lambda_1,\mu,\nu_1}(q)}\sum_{w_2\in W_2^{\lambda_2,\mu,\nu_2}(q)} p_{w_1}p_{w_2}\chi_q  \|^p_{L^p(b)}	\\
	&\le & \sum_b I^{\sim b} + \sum_b I^{\not\sim b},
\end{eqnarray*}
where  $I^{\sim b}$ denotes the contribution to  $\|\sum_{q\in Q^\mu}\sum_{w_1\in W_1^{\lambda_1,\mu,\nu_1}(q)}\sum_{w_2\in W_2^{\lambda_2,\mu,\nu_2}(q)} p_{w_1}p_{w_2}\chi_q  \|^p_{L^p(b)}$ by all  $w_1\sim b$ and $w_2\sim b,$ and $I^{\not\sim b}$ the contribution by all $w_1$ and $w_2$ such that $w_1\not\sim b$ or $w_2\not\sim b.$ 

For any fixed $b$, as in the classical bilinear argument, the part $I^{\sim b}$  can  be estimated applying the induction hypothesis $E(\alpha)$ on the cube $b,$ which is a translate of $Q(R'^{1-2\gamma})$. Then the sum over  all $b$'s can eventually be controlled since for any given $w_j$ there are only $\mathcal{O}(1)$ $b's$ such that  $b\sim w_j,$  and we obtain 
\begin{equation}\label{simb}
	\sum_b I^{\sim b} \lesssim C_{p,\alpha} R'^{\alpha(1-2\gamma)} \, \delta^{\frac{5}{2}-\frac{4}{p}}
	|W_1|^{1/2}|W_2|^{1/2}.
\end{equation}
We omit the details, since this part is still standard. 

\medskip

In order to estimate the remainder part $I^{\not\sim b}$, let us define
$$
	W_j^{\lambda_j,\mu,\not\sim b}:=\{w_j\in W_j^{\lambda_j,\mu}:w_j\not\sim b\}.
$$
Then it remains to prove that for some absolute constant $c>0$ independent of $R,W_1,W_2,\delta,\gamma$ and  $\la, \mu, \nu$ and any $b$ we have an estimate of the form
\begin{align}\label{waveest2}
	\|\sum_{q\in Q^\mu}\sum_{w_1\in W_1^{\lambda_1,\mu,\nu_1,\not\sim b}(q)}\sum_{w_2\in W_2^{\lambda_2,\mu,\nu_2}(q)} p_{w_1}p_{w_2}\chi_q \|^p_{L^p(b)}
	\leq C_{p,\alpha} R'^{c\gamma} \, \delta^{\frac{5}{2}-\frac{4}{p}}
	|W_1|^{1/2}|W_2|^{1/2},
\end{align}
and an analogous  estimate  with the roles of $W_1$ and $W_2$ interchanged.
\medskip

Indeed,  given these estimates, it is easy to complete the proof of Theorem \ref{bilinear}.  To this end,  notice that summing these estimates over  all $b$ merely increases the constant $c$ in the exponent of $R'^{c\delta},$ since there are only $R'^{\Landau(\gamma)}$ cubes $b$. Combining then \eqref{simb} and \eqref{waveest2}, we obtain
\begin{align}
	\|\sum_{q\in Q^\mu}\sum_{w_1\in W_1^{\lambda_1,\mu,\nu_1}(q)}\sum_{w_2\in W_2^{\lambda_2,\mu,\nu_2}(q)} p_{w_1}p_{w_2}\chi_q  \|_{L^p(Q(R))}
	\leq C_{p,\alpha} R'^{\alpha(1-2\gamma)\vee c\gamma} \delta^{\frac{5}{2}-\frac{4}{p}}
	|W_1|^{1/2}|W_2|^{1/2}.
\end{align}
Hence, by Lemma \ref{reduct}, we see that $E(\alpha)$ implies $E(\alpha(1-\gamma)\vee c\gamma)$ (with possibly yet another constant $c$). 

We should mention here that the base case for the induction on scales is not that straight forward. Usually some crude estimate with a big loss in the power of $R'$ will suffice, but we need the sharp dependency on $\delta$ in \eqref{locbilinest}. However, our argument from \cite{bmv16} and \cite{bmv17}  does apply also here (compare Remark \ref{basecase}).
 Thus,  by induction, we see that $E(\alpha)$ holds for arbitrary small $\alpha$, which proves Theorem \ref{bilinear}, provided we have  established \eqref{waveest2}.

\subsection{The geometric argument}
A key element of the bilinear method is a sophisticated geometric argument.  It allows to estimates the number of tubes whose directions move according to  normals to the surface  along certain points on an intersection curve of the surfaces $S_1$ and $S_2$. By an intersection curve, we mean the following:  the patches of surface $S_1$ and $S_2$ are disjoint, but  by means of suitable translations we may achieve that they do intersect along a curve.  

More precisely,  for any $v'_1\in\V_1$ and  $v'_2\in\V_2,$  we define  the intersection curve $\Pi_{v'_1,v'_2}:= \big(S_1-(v'_1,\phi(v'_1))\big)\cap\big(S_2-(v'_2,\phi(v'_2))\big).$ Moreover, if  $W_1\subset\W_1$,  we put 
$$
[W_1]^{\Pi_{v'_1,v'_2}}:=\{w_1=(\eta_1,v_1)\in W_1: (v_1-v_1',\phi(v_1)-\phi(v_1'))\in \Pi_{v'_1,v'_2}+\mathcal{O}(R^{-1})\}.
$$
Note that  it is only the component $v_1$ of $w_1$ which is relevant for $w_1$ to be in  $[W_1]^{\Pi_{v'_1,v'_2}}.$ We therefore also introduce the projection of this set to all of its $v_1$-components, i.e., 
$$
[W_1]^{\Pi_{v'_1,v'_2}}_{\V_1}:=\{v_1\in \V_1: \exists \eta_1 \text{ such that } (\eta_1,v_1)\in [W_1]^{\Pi_{v'_1,v'_2}}\}\subset \V_1.
$$

There is an a priori estimate on the cardinality of $[W_1]^{\Pi_{v'_1,v'_2}}_{\V_1}$: As the projection of any intersection curve $(S_1-(v'_1,\phi(v'_1)))\cap (S_2-(v'_2,\phi(v'_2)))$ to the $z$-plane intersects the set $U_1-v_1'$ diagonally (compare \eqref{tangent}), that intersection curve has length $\lesssim\delta$. Since  by \eqref{grad1} the directions of the tubes are $1/R$-separated, we have for all $W_1\subset\W_1$
\begin{align}\label{aprioriintersect}
	\sup_{v_1'\in\W_1,v_2'\in\V_2}\Big |[W_1]^{\Pi_{v'_1,v'_2}}_{\V_1}\Big |\lesssim \delta R=R'.
\end{align}

 In a similar way, we define the sets $[W_2]^{\Pi_{v'_1,v'_2}}\subset W_2$ and $[W_2]^{\Pi_{v'_1,v'_2}}_{\V_2}\subset \V_2$  simply by interchanging the roles of $S_1$ and $S_2,$ and obtain the analogue  to \eqref{aprioriintersect} for $\Big |[W_2]^{\Pi_{v'_1,v'_2}}_{\V_2}\Big |.$
 
 If we consider the ``cone''  generated by the  family of all tubes $T_{w_1}$ which pass through some fixed  cube $q_0$ and whose directions are  given by the normals to $S_1$ at all points $(v_1,\phi(v_1))$  with $v_1\in [W_1]^{\Pi_{v'_1,v'_2}}_{\V_1}$ (for some $v_1',v_2'$), then  the geometric meaning of our transversality assumption $\eqref{tv3012}$ is exactly that all tubes $T_{w_2}$ of type 2  pass transversally through this cone. The  analogous statement holds true  for the ``cone''  generated by the  family of all tubes $T_{w_2}$ which pass through some fixed cube $q_0$ and whose directions are  given by the normals to $S_2$ at all points $(v_2,\phi(v_2))$  with $v_2\in [W_2]^{\Pi_{v'_1,v'_2}}.$

\begin{figure}
\begin{center}  \includegraphics{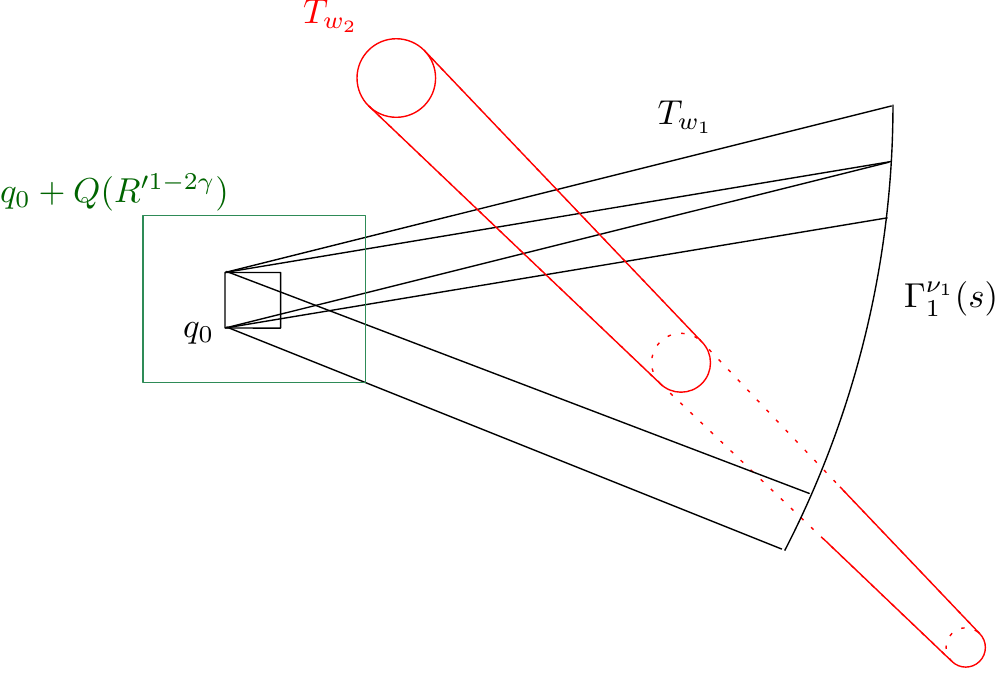} \end{center}
\caption{The geometry in Lemma \ref{geom}}
\label{cone}       
\end{figure}

 \smallskip
 
These facts are crucial for the following lemma.
\color{black}

\begin{lemma}\label{geom}
For all dyadic $\mu$, $\lambda$, $1\le \nu_1,\nu_2< \nu_{\rm max}$, all  cubes $q_0$ and $v'_1\in\V_1,v'_2\in\V_2$ we have
\begin{eqnarray}\label{eq:geom}
\lambda_{\nu_1,1}\mu_{\nu_2,2} \Big |[W_1^{\lambda_1,\mu,\nu_1,\not\sim b}(q_0)]^{\Pi_{v'_1,v'_2}}\big |
	&\lesssim& {R'}^{c\gamma} \nu_1^2\nu_2^2 |W_2|,	\\
	\lambda_{\nu_2,2}\mu_{\nu_1,1} \big |[W_2^{\lambda_2,\mu,\nu_2,\not\sim b}(q_0)]^{\Pi_{v'_1,v'_2}}\big |
	&\lesssim& R'^{c\gamma} \nu_1^2\nu_2^2 |W_1|. \label{eq:geom2}
\end{eqnarray}
\end{lemma}
\begin{remark} The classical version of this argument corresponds to the case $\nu_1=\nu_2=1$, cf. \cite{lee05}.
\end{remark}

\begin{proof}
We shall only prove \eqref{eq:geom}, the proof of  \eqref{eq:geom2} being analogous.

Let us  consider the ``cone''  generated  by the tubes $T_{w_1}^{\nu_1}$  associated to all   $w_1\in [W_1^{\lambda_1,\mu,\nu_1,\not\sim b}(q_0)]^{\Pi_{v_1,v_2}}$, however, without the part near $q_0$ where the tubes overlap. More precisely, we define
$$
\Gamma_1^{\nu_1}:=\bigcup\{T_{w_1}^{\nu_1}: w_1\in [W_1^{\lambda_1,\mu,\nu_1,\not\sim b}(q_0)]^{\Pi_{v'_1,v'_2}}\}\setminus(q_0+Q(R'^{1-2\gamma}))
$$
and the corresponding set of cubes $q$ by
$$
Q^{\nu_1}_{\Gamma_1} := \{q\in Q^\mu: q\cap \Gamma_1^{\nu_1}\neq\emptyset\}.
$$
Since the directions of the tubes are only allowed to move as the normals to $S_1$ corresponding to   points on the  intersection curve $\Pi_{v'_1,v'_2}$, the set $\Gamma_1^{\nu_1}$ is indeed a part of a cone thickened by $\nu_1R R'^{\gamma}$. We may decompose  $\Gamma_1^{\nu_1}$ into disjoint layers  $\Gamma_1^{\nu_1}(s)$, $s=1,\ldots,\nu_1,$ which are  essentially translates of $\Gamma_1^{1}$, that is, parts of a cone thickened by $R R'^\gamma$. Correspondingly, let
$$Q^{\nu_1}_{\Gamma_1}(s) := \{q\in Q^\mu: q\cap \Gamma_1^{\nu_1}(s)\neq\emptyset\}, \qquad s=1,\ldots,\nu_1.$$
The proof will be accomplished  by verifying the following three estimates:
\medskip

\begin{enumerate}
\item \quad $\lambda_{\nu_1,1}\Big|[W_1^{\lambda_1,\mu,\nu_1,\not\sim b}(q_0)]^{\Pi_{v'_1,v'_2}}\Big|
		\lesssim R'^{c\gamma}\nu_1 |Q^{\nu_1}_{\Gamma_1}|:$
\item\quad  $|Q^{\nu_1}_{\Gamma_1}| \leq \sum\limits_{s=1}^{\nu_1} |Q^{\nu_1}_{\Gamma_1}(s)|
		\leq {\nu_1} \sup\limits_s|Q^{\nu_1}_{\Gamma_1}(s)|;$
\item \quad $\mu_{\nu_2,2}\, \sup_s|Q^{\nu_1}_{\Gamma_1}(s)|\lesssim R'^{c\gamma}\nu_2^2|W_2|.$
\end{enumerate}

\smallskip

It is obvious that (i) -- (iii) imply estimate  \eqref{eq:geom}.  Estimate  (ii) is clear too.
\smallskip

To verify  (i),  we observe that $\lambda_{\nu_1,1} \sim |\{q\in Q^\mu: T_{w_1}^{\nu_1}\cap q\neq\emptyset\}|$ for any $w_1\in W_1^{\lambda_1,\mu}$, hence by Fubini
\begin{align*}
	\lambda_{\nu_1,1}\Big|[W_1^{\lambda_1,\mu,\nu_1,\not\sim b}(q_0)]^{\Pi_{v'_1,v'_2}}\Big|
	\lesssim&  \sum_{w_1\in[W_1^{\lambda_1,\mu,\nu_1,\not\sim b}(q_0)]^{\Pi_{v'_1,v'_2}}} 
													|\{q\in Q^\mu: T_{w_1}^{\nu_1}\cap q\neq\emptyset\}|	\\
	\lesssim& \sum_{q\in Q^{\nu_1}_{\Gamma_1}}
					|\{w_1\in[W_1^{\lambda_1,\mu,\nu_1,\not\sim b}(q_0)]^{\Pi_{v'_1,v'_2}}: 
									T_{w_1}^{\nu_1}\cap q\neq\emptyset\}|	\\
	\lesssim& R'^{c\gamma} |Q^{\nu_1}_{\Gamma_1}|\nu_1 .																					
\end{align*}
The last estimate holds true because while the tubes $T_{w_1}$ passing through $q_0$ with directions given by normals to the surface
when moving along the intersection curve overlap at most of order $\mathcal{O}(R'^{c\gamma})$, our hollow thickened tubes $T_{w_1}^{\nu_1}$ overlap at most  of order $\mathcal{O}(\nu_1R'^{c\gamma})$.
\smallskip

It remains to prove  (iii).  To this end, let us fix an $s\in \{1,\dots, \nu_1\},$ i.e., a certain layer $\Gamma_1^{\nu_1}(s)$ of our  thickened cone. Note next that every  tube $T_{w_2}$ of type 2 is   transversal to the cone $\Gamma_1^{\nu_1},$  hence also to $\Gamma_1^{\nu_1}(s),$  and it  essentially cuts out a ``cap'' of radius $\nu_2R R'\gamma$ and thickness 
$R R'^\gamma$ from $\Gamma_1^{\nu_1}(s)$ (cf. picture 1).

 We  therefore decompose  $\Gamma_1^{\nu_1}(s)$ into caps  $P$ of dimensions  $\nu_2R R'\gamma\times \nu_2R R'\gamma\times R R'^\gamma.$ Of course we only need to consider those caps $P$ which interact with at least one of the $q$'s, and therefore   put
$$\mathfrak{P}^\mu_{\Gamma_1}:=\{\text { caps } P\subset\Gamma_1^{\nu_1}(s):\exists q\in Q^\mu:q\cap P\neq\emptyset\}.$$
Then
\begin{equation}\label{Qnuest}
|Q^{\nu_1}_{\Gamma_1}(s)|\le \sum_{P\in\mathfrak{P}^\mu_{\Gamma_1}} |\{q\in Q^{\nu_1}_{\Gamma_1}(s)|q\cap P\neq\emptyset\}| 
	\leq 
	\nu_2^2 R'^{2\gamma}|\mathfrak{P}^\mu_{\Gamma_1}|.
\end{equation}

On the other hand, for any $q\in Q^\mu$ such that $q\cap P\neq\emptyset$, we have 
$$\mu_{\nu_2,2}\sim |\{w_2\in W_2:T_{w_2}^{\nu_2}\cap q\neq\emptyset\}|\leq |\{w_2\in W_2:T_{w_2}^{\nu_2}\cap 2P\neq\emptyset\}|.
$$
Therefore by \eqref{Qnuest} and  Fubini's theorem
\begin{eqnarray*}
\mu_{\nu_2,2}|Q^{\nu_1}_{\Gamma_1}(s)| &\leq& \nu_2^2  R'^{2\gamma} \mu_{\nu_2,2} |\mathfrak{P}^\mu_{\Gamma_1}|	\\
	&\lesssim& \nu_2^2 R'^{2\gamma}\sum_{P\in\mathfrak{P}^\mu_{\Gamma_1}} |\{w_2\in W_2:T_{w_2}^{\nu_2}\cap 2P\neq\emptyset\}|		\\
	&\leq& \nu_2^2 R'^{2\gamma}\sum_{w_2\in W_2} |\{P\in\mathfrak{P}^\mu_{\Gamma_1}:T_{w_2}^{\nu_2}\cap 2P\neq\emptyset\}|	\\
	&\lesssim& R'^{c\gamma}\nu_2^2 |W_2|.
\end{eqnarray*}
The last inequality holds because the tubes $T_{w_2}$ are transversal to the cone $\Gamma_1^{\nu_1}.$
This  verifies (iii). 
\end{proof}

\subsection{Completion of the proof of Theorem \ref{bilinestn} }
We were still left with proving \eqref{waveest2}, i.e., the estimate
$$
	\|\sum_{q\in Q^\mu}\sum_{w_1\in W_1^{\lambda_1,\mu,\nu_1,\not\sim b}(q)}\sum_{w_2\in W_2^{\lambda_2,\mu,\nu_2}(q)} p_{w_1}p_{w_2}\chi_q \|^p_{L^p(b)}
	\leq C_{p,\alpha} R'^{c\gamma} \, \delta^{\frac{5}{2}-\frac{4}{p}}
	|W_1|^{1/2}|W_2|^{1/2}.
$$
As usual, this is accomplished by interpolating an $L^1$ and an $L^2$-estimate. The $L^1$-estimate is the usual straight-forward application of Cauchy-Schwarz' inequality:
\begin{eqnarray}
&&\big\|\sum_{q\in Q^\mu}\sum_{w_1\in W_1^{\lambda_1,\mu,\nu_1,\not\sim b}(q)}\sum_{w_2\in W_2^{\lambda_2,\mu,\nu_2}(q)}  p_{w_1}p_{w_2}\chi_q \big\|_{L^1(b)}\nonumber \\
	&\leq& \big\|\sum_{w_1\in W_1^{\lambda_1,\mu,\not\sim b}} p_{w_1}\big\|_{L^2(b)}\cdot\big\|\sum_{w_2\in W_2^{\lambda_2,\mu}} p_{w_2}\big\|_{L^2(b)} 
	\nonumber \\
	&\lesssim& (R^2\cdot\delta R^2)^{1/2}\ |W_1|^{1/2}|W_2|^{1/2}	\nonumber\\
	&=&\delta^{-3/2} R'^2\ |W_1|^{1/2}|W_2|^{1/2}.\label{ellone}
\end{eqnarray}

	For the $L^2$-estimate, in order to defray the notation,  let us abbreviate $W_1^0(q):=W_1^{\lambda_1,\mu,\nu_1,\not\sim b}(q)$, $W_2^0(q):=W_2^{\lambda_2,\mu,\nu_2}(q)$.
	
	 Recall also  that $\chi_q$ is not the characteristic function of $q$, but a smooth version of it which  decays rapidly away from $q$. Nevertheless, Schur's test (cf. the similar argument on p. 856 in \cite{bmv16}) implies that 
$$
	\big\|\sum_{q\in Q^\mu}\sum_{w_1\in W_1^0(q)}\sum_{w_2\in W_2^0(q)}  p_{w_1}p_{w_2}\chi_q \big\|^2_{L^2(b)}
	\lesssim \sum_{q\in Q^\mu} \big\|\sum_{w_1\in W_1^0(q)}\sum_{w_2\in W_2^0(q)}  p_{w_1}p_{w_2}\chi_q \big\|^2_{L^2(b)}.
$$
We proceed in the usual way by applying Plancherel's theorem  in order to exploit information on the Fourier support of the wave packets. It is here where  we need \eqref{pdos} to keep track of the localization by $\chi_q:$
\begin{eqnarray*}
&&\big\|\sum_{q\in Q^\mu}\sum_{w_1\in W_1^{0}(q)}\sum_{w_2\in W_2^0(q)}  p_{w_1}p_{w_2}\chi_q \big\|^2_{L^2(b)}	\\
  &\lesssim& \sum_{q\in Q^\mu}
   \sum_{w_1,w_1'\in W_1^0(q)}\sum_{w_2,w_2'\in W_2^0(q)}
  \int  p_{w_1}p_{w_2'}\overline p_{w_1'}\overline p_{w_2}\chi_q	\\
 & \lesssim& \sum_{q\in Q^\mu}\ \sum_{\substack{w_1\in W_1^0(q)\\w_2\in W_2^0(q)}}\
  \sum_{v_1'\in[W_1^0(q)]^{\Pi_{v_1,v_2}}}  |q|\,  \Big  \|p_{w_1}p_{w_2}	\sum_{\underset{(\eta_1',v'_1)\in W_1^0(q)}{\eta_1':}}\sum_{\underset{v_1+v_2'=v_1'+v_2}{w_2'\in W_2^0(q):}} p_{w'_1}p_{w'_2}\chi_q \Big \|_{L^\infty}.
 \end{eqnarray*}
 Note also that 
 \begin{eqnarray} \label{inftyest}
&& \hskip1cm \Big  \|p_{w_1}p_{w_2}	\sum_{\underset{(\eta_1',v_1)\in W_1^0(q)}{\eta_1':}}\sum_{\underset{v_1+v_2'=v_1'+v_2}{w_2'\in W_2^0(q):}} p_{w'_1}p_{w'_2}\chi_q \Big \|_{L^\infty} \\
&\le& \big\| p_{w_1} |\chi_q|^{1/4}\big\|_{L^\infty} \, \big\| p_{w_2} |\chi_q|^{1/4}\big\|_{L^\infty} \, \big\| \sum_{\underset{w_1'=(\eta_1',v'_1)\in W_1^0(q)}{\eta_1':}} p_{w'_1}|\chi_q|^{1/4}\big\|_{L^\infty} \, 
\big\|\sum_{\underset{v_2'=v_1'+v_2-v_1}{w_2'\in W_2^0(q):}}p_{w'_2} |\chi_q|^{1/4}\big\|_{L^\infty}.\nonumber
\end{eqnarray} 
Moreover,  if $w_1'\in W_1^0(q),$ then  $T^{\nu_1}_{w'_1} \cap q\ne \emptyset, $ and thus $\dist (T_{w'_1}, q)\gtrsim \nu_1 R R'^\gamma\ge \nu_1 R$ provided  $\nu_1>1,$  because  then 
$T^{\nu_1}_{w'_1}$ is a hollow tube. Lemma \ref{franz} (i)  with $d_1:=\nu_1$  (actually with $|\chi_q|^{1/4}$ in place  of $\chi_q$) then implies  that 
$$
\big\| \sum_{\underset{w_1'=(\eta_1',v'_1)\in W_1^0(q)}{\eta_1':}} p_{w'_1}|\chi_q|^{1/4}\big\|_{L^\infty} \lesssim \nu_1^{-1} R^{-1}.
$$
If $\nu_1=1,$ we can still apply Lemma \ref{franz} (i), however now with $d_1:=0,$ and arrive at the same estimate.  In a similar way,  Lemma \ref{franz} (ii)  allows to estimate the last factor in \eqref{inftyest}, and altogether we find  that 
$$
 \Big  \|p_{w_1}p_{w_2}	\sum_{\underset{(\eta_1',v_1)\in W_1^0(q)}{\eta_1':}}\sum_{\underset{v_1+v_2'=v_1'+v_2}{w_2'\in W_2^0(q):}} p_{w'_1}p_{w'_2}\chi_q \Big \|_{L^\infty}\lesssim \nu_1^{-2} \nu_2^{-2} R^{-4}.
 $$
We thus conclude that  
\begin{eqnarray}\label{03jan1302}
&&\big\|\sum_{q\in Q^\mu}\sum_{w_1\in W_1^{0}(q)}\sum_{w_2\in W_2^0(q)}  p_{w_1}p_{w_2}\chi_q \|^2_{L^2(b)}	\nonumber\\
	&\lesssim& R^{-1} \nu_1^{-2}\nu_2^{-2} \sum_{q\in Q^\mu} |W_1^{\lambda_1,\mu,\nu_1,\not\sim b}(q)|\ 
							|W_2^{\lambda_2,\mu,\nu_2}(q)|\ \sup_{w_1,w_2}\Big|[W_1^{\lambda_1,\mu,\nu_1,\not\sim b}(q)]^{\Pi_{v_1,v_2}}\Big|.
\end{eqnarray}

In the case that $\nu_1,\nu_2< \nu_{\rm max}$, we can further apply the geometric Lemma \ref{geom}:
\begin{eqnarray*}
&&\big\|\sum_{q\in Q^\mu}\sum_{w_1\in W_1^{\lambda_1,\mu,\nu_1,\not\sim b}(q)}\sum_{w_2\in W_2^{\lambda_2,\mu,\nu_2}(q)}  p_{w_1}p_{w_2}\chi_q \big\|^2_{L^2(b)}	\\
	&\lesssim& {R'}^{c\gamma}R^{-1} \sum_{q\in Q^\mu} |W_1^{\lambda_1,\mu,\nu_1,\not\sim b}(q)|\frac{|W_2^{\lambda_2,\mu,\nu_2}(q)|}{\mu_{\nu_2,2}} \frac{|W_2|}{\lambda_{\nu_1,1}}	\\
	&\lesssim& {R'}^{c\gamma}R^{-1} \sum_{w_1\in W_1^{\lambda_1,\mu,\nu_1,\not\sim b}} \frac{|Q^{\mu,\nu_1}(w_1)|}{\lambda_{\nu_1,1}} |W_2|	\\
	&\lesssim& {R'}^{c\gamma}R^{-1} |W_1|\,|W_2|,
\end{eqnarray*}
since 
$|W_2^{\lambda_2,\mu,\nu_2}(q)|/\mu_{\nu_2,2}\sim 1\sim  |Q^{\mu,\nu_1}(w_1)| /\lambda_{\nu_1,1}.$ In the second last inequality, we have again applied Fubini's theorem.

In the case that $\nu_1= \nu_{\rm max}$ or $\nu_2= \nu_{\rm max}$, we can use simpler estimates. Note that $\nu_1\vee\nu_2= \nu_{\rm max}\sim {R'}^{1-2\gamma} $,
whereas by \eqref{aprioriintersect} we have  $\Big|[W_1^{\lambda_1,\mu,\nu_1,\not\sim b}(q)]_{\V_1}^{\Pi_{v_1,v_2}}\Big|\lesssim R'$.
Hence \eqref{03jan1302} and Fubini's theorem  give
\begin{eqnarray*}
	&&\big\|\sum_{q\in Q^\mu}\sum_{w_1\in W_1^{0}(q)}\sum_{w_2\in W_2^0(q)}  p_{w_1}p_{w_2}\chi_q \big\|^2_{L^2(b)}	\\
	&\lesssim& R^{-1} \nu_1^{-2}\nu_2^{-2} \sum_{w_1\in W_1}\sum_{w_2\in W_2} |\{q\in Q^\mu:q\cap T_{w_1}^{\nu_1}\neq\emptyset\neq q\cap T_{w_2}^{\nu_2}\}|\ R' \\
	&\lesssim& R^{-1} |W_1|\ |W_2|\ \nu_1^{-1}\nu_2^{-1} (\nu_1\wedge\nu_2)\ R'	\\
	&\lesssim& R'^{c\gamma} R^{-1} |W_1|\ |W_2|. 
\end{eqnarray*}
For the second last inequality, we used that $|T_{w_1}^{\nu_1}\cap T_{w_2}^{\nu_2}| \sim \nu_1^2 \nu_2 |q|$ if, say, $\nu_1\le \nu_2.$ 
Altogether, in both cases we obtain the estimate
\begin{eqnarray*}
&&\big\|\sum_{q\in Q^\mu}\sum_{w_1\in W_1^{\lambda_1,\mu,\nu_1,\not\sim b}(q)}\sum_{w_2\in W_2^{\lambda_2,\mu,\nu_2}(q)}  p_{w_1}p_{w_2}\chi_q \big\|_{L^2(b)} \\
	&\lesssim& {R'}^{c\gamma} R^{-1/2} |W_1|^{1/2}|W_2|^{1/2}  \nonumber\\
	&=& \delta^{1/2} R'^{-1/2}{R'}^{c\gamma} |W_1|^{1/2}|W_2|^{1/2}.
\end{eqnarray*}

Interpolating this with the $L^1$ estimate \eqref{ellone}, we obtain $\eqref{waveest2},$ provided $p>5/3$.

\begin{remark}\label{basecase} For the base case of our induction, i.e., \eqref{locbilinest} with any big power $\alpha$ of $R'$, it is again enough to estimate the corresponding wave packets
$$\|\sum_{w_1\in W_1}\sum_{w_2\in W_2}p_{w_1}p_{w_2} \|_p \lesssim {R'}^{\alpha} |W_1|^{1/2}|W_2|^{1/2}.$$
The $L^1$ estimate is similar as before, and for the $L^2$-estimate we can begin similarly as above, using Lemma \ref{franz} (i), (ii) and now also (iii) to see that 
\begin{eqnarray*}
	&&\big\|\sum_{w_1\in W_1}\sum_{w_2\in W_2}  p_{w_1}p_{w_2} \big\|^2_{2}	\\
	&\lesssim& 
   \sum_{\substack{w_1\in W_1\\w_2\in W_2}}\sum_{v_1'\in[W_1]^{\Pi_{v_1,v_2}}}\sum_{\underset{v_1+v_2'=v_1'+v_2}{v_2':}}
  \int  |p_{w_1} p_{w_2}| \,d(\xi,\tau)\	\Big\|\sum_{\eta'_1} p_{w_1'}\Big\|_{L^\infty} \Big\|\sum_{\eta'_2} p_{w_2'}\Big\|_{L^\infty} \\	
	&\lesssim& R^{-1} |W_1||W_2|\ \sup_{v_1,v_2}\Big|[W_1]_{\V_1}^{\Pi_{v_1,v_2}}\Big|.
\end{eqnarray*}
Applying \eqref{aprioriintersect} now and interpolating with the $L^1$ estimate gives the correct power of $\delta$.
\end{remark}

\color{black}



\setcounter{equation}{0}
\section{Passage to linear restriction estimates  and proof of Theorem \ref{mainresult}}\label{bilinlin}
To prove Theorem \ref{mainresult}, assume that $r>10/3$ and $1/q'>2/r$ and put $p:=r/2,$
so that $p>5/3$, $1/q'>1/p.$ By interpolation with the trivial estimate for
$r=\infty,q=1,$
it is enough to prove the result  for $r$ close to $10/3$ and  $q$ close to 5/2, i.e.,
$p$
close to $5/3$ and $q$ close to 5/2. Hence, we may  assume that $p<2,$ $p<q<2p=r.$  Recall also that we  already assume that $h(y)$ is flat at the origin, that $h'''(y)>0$ for $y>0,$  and 
 that $\supp f\subset\{(x,y)\in \Om:  y\ge 0\}.$  Let let us therefore in the sequel assume that $\Om$ is a rectangle of the form 
 $\Om=[-1,1]\times [0,\eta_1],$ where $\eta_1>0$ is a sufficiently small dyadic number. Corresponding to $\eta_1,$ we also choose a small dyadic number $\la_1>0$ so that $|h'''(y)|\le \la_1$ on $\Om.$ Note that we can choose $\la_1$ sufficiently small by choosing $\eta_1$ small enough. 
 
The problem in passing  from our restriction  estimates in Theorem \ref{cubrest} to the result in Theorem \ref{mainresult} lies in the fact that, unlike in the case of finite type perturbations studied in \cite{bmv17}, we can here usually no longer simply  sum the estimates given by Theorem \ref{cubrest} for the contributions by the horizontal strips  on which $h'''\sim \la$ over all dyadic values $0<\la \le \la_1,$ since we have insufficient control  on the lengths of the intervals $I_{\la,\iota}$ given by Theorem \ref{levelset}. To overcome this problem, we shall apply once more the bilinear method,  making use of bilinear estimates for pairs of surface patches on which $h'''$ is sufficiently small, which, however, will rather easily be established, following basically the approach from \cite{v05} (and \cite{lee05}) devised for the unperturbed  parabolic hyperboloid. In a final step, we shall then fuse these  bilinear estimates with the linear estimates from Theorem \ref{cubrest} by means of a kind of bootstrap argument. It will be in this part where we shall have to make the stronger assumption that $h'''$ is monotonic (say, increasing).

\smallskip

A key step in  \cite{v05}, \cite{lee05} consists in devising  a suitable  Whitney-type decomposition of $\Om\times\Om$ into direct products of rectangles; more details 
will be given later.  This leads us to considering bilinear estimates over pairs of ``close'' rectangles $W_1=J_1\times I_1,W_2=J_2\times I_2$  contained in $\Om,$ where $J_1,J_2$ are intervals of dyadic length $d\le 1,$ and  $I_1,I_2$ are intervals of dyadic length $r\le \eta_1,$ and which are separated in the $x$-coordinate  of size $d$ and in the $y$-coordinates of size $r.$ 

Consider such a pair $W_1,W_2,$ and set for $i=1,2,$
\begin{eqnarray*}
W_i^>:&=&\{(x,y)\in W_i: h'''(y) >dr/100\},\\
W_i^<:&=&\{(x,y)\in W_i: h'''(y) \le dr/100\}.
\end{eqnarray*}
By applying  Theorem \ref{levelset} to the function $\vp:=h'''$  let us decompose the interval $[0,\eta_1]$ into intervals $I_{\la,\iota}$ on which 
$\la/2<h'''(y)\le 4\la \quad (\la>0 \quad\text{dyadic}).$
Note that if $h'''$ is monotonic, for each $\la$ there will be at most one corresponding $\iota,$ but for the time being this will not yet be relevant and we could as well work with non-monotic $h'''.$ Correspondingly, for any subset $W$ of $\Om$ we put $W^{\la,\iota}:= \{(x,y)\in W: y\in I_{\la,\iota}\}.$ 
\medskip

Choose next $\tilde q$ so that $1/q'>1/\tilde q'>2/r.$ Applying the argument from the beginning of Section \ref{h'''simla} and making use of the equivalence of \eqref{extI1} and \eqref{extI2} in combination with Theorem \ref{cubrest} we then find that
$$
\|\ext_{W_i^{\la,\iota}}f\|_{L^r} \leq C_{r,\tilde q} \|f|_{W_i^{\la,\iota}}\|_{L^{\tilde q}},
$$
since $d_{I_{\la,\iota}}^{1-2/r-1/\tilde q}\le 1,$  and thus in combination with H\"older's estimate we obtain 
 $$
\|\ext_{W_i^{\la,\iota}}f\|_{L^r} \leq C_{r,\tilde q}(dr)^{\frac 1{\tilde q}-\frac 1q} \|f|_{W_i}\|_{L^q}.
$$
Write $\ve:=( 1/{\tilde q}- 1/q)/3.$  Then $\e>0,$  and if $\la\ge dr/400,$ then $(dr)^{\frac 1{\tilde q}-\frac 1q}\lesssim \la^{2\ve} (dr)^\ve.$  Decomposing $W_1^>$ into such sets $W_i^{\la,\iota},$  we then see that by choosing $r\in\N$ in Theorem \ref{levelset} sufficiently large we can sum the preceding estimates over the corresponding dyadic $\la$'s with $\la\le \la_1$ and $\iota$'s and arrive at the following uniform linear estimates:

\begin{equation}\label{E>} 
\|\ext_{W_i^>}f\|_{L^r} \leq C_{r,q} \la_1^\ve  |W_i|^\ve\|f_{W_i^>}\|_{L^q},\qquad i=1,2,
\end{equation}
where have used the notation  $f_A:= f\chi_A$ for any subset $A\subset \Om.$ Recall that this part of the argument does not require any monotonicity assumption on $h'''$ yet.

\medskip
Our next step, namely the proof of the following bilinear estimate for the operators $\ext_{W_i^<},$ will, in contrast, make use of this monotonicity assumption:

\begin{equation}\label{E<}
\|\ext_{W_1^<}f\, \ext_{W_2^<}g\|_{L^p} \leq C_{r,q}  |W_1|^{\frac 1 {q'}-\frac 1 p} |W_2|^{\frac 1 {q'}-\frac 1 p}\|f_{W_1^<}\|_{L^q} \|g_{W_2^<}\|_{L^q}.
\end{equation}
\begin{proof} 
By the monotonicity of $h''',$ the set $J^<:=\{y\in [0,\la_1]: h'''(y)\le dr/100\}$ is an interval of the form $J^<=[0,\eta],$ with $0\le \eta\le \eta_1.$ Looking at the Taylor expansion of $h$ around $0,$  we see that this implies that 
$$
|h^{(k)}(y)| \le dr/100 \qquad \text{for all} \quad y\in J^<,\,  k=0,1,2,3.
$$
Let us next assume w.l.o.g. that the rectangle $W_1$ is located ``below''  $W_2,$ and that $W_2^<\ne \emptyset.$ Then 
$$
W_1^<=W_1 \quad \text{and}\quad W_2^<=W_2\cap \{(x,y):y\in [0,\eta]\}.
$$
Denote by $z_1^0=(x_1^0,y_1^0)$ the lower left vertex of $W_1,$ and translate the coordinates $(x,y)$ so that $z_1^0$ becomes the origin. Then, in the new coordinates, the function $\phi$ assumes the form
$$
\phi(x,y)=xy+H(y),
$$
except for affine-linear terms which have no effect on the bilinear estimates, and $W_1$ assumes the form $W_1=[0,d]\times [0,r],$ and $W_2^<$ the form 
 $W_2^<=[2d, 3d]\times  [2r,\tilde \eta],$ with $2r\le \tilde \eta\le 3r,$ or $W_2^<=[-2d, -d]\times  [2r,\tilde \eta].$ Moreover, $H$ will also satisfy the estimates
 \begin{equation}\label{Hderiv}
|H^{(k)}(y)| \le dr/100 \qquad \text{for all} \quad y\in [0,\tilde \eta],\,   k=0,1,2,3.
\end{equation}
We next pass to the re-scaled coordinates $(x',y')$ defined by $x=d x', \, y=ry',$  and put 
$$
\phi^s(x',y'):= \frac 1{dr}\phi(dx',ry')= x'y'+H^s(y'), 
$$
where $H^s(y'):=H(ry')/dr.$ The rectangles $W_1^<$ and $W_2^<$ correspond to $(W_1^<)^s:=[0,1]\times [0,1]$ and  $W_2^<=[2, 3]\times  [2,\eta^s],$ or
$(W_2^<)^s=[-2, -1]\times  [2,\eta^s],$  in these coordinates, with $2\le \eta^s\le 3.$ Finally note that by \eqref{Hderiv}, we have 
 \begin{equation}\label{Hsderiv}
|(H^s)^{(k)}(y')| \le 1/100 \qquad \text{for all} \quad y'\in [0,\eta^s],\,   k=0,1,2,3,
\end{equation}
so that the function $\phi^s(x',y')$ is a very small perturbation of $x'y',$ whereas $(W_1^<)^s$ and $(W_2^<)^s$ are contained in squares of side length $1$ which are $1$-separated in each coordinate. One can then easily check that both  transversality functions $|TV^s_1|$ and $|TV^s_2|$ associated to $\phi^s$  (cf. \eqref{transnew}) are of size $\sim 1$ on $(W_1^<)^s\times (W_2^<)^s.$ Thus, if we had also a good control on higher order derivatives of $H^s$ of the form
$|(H^s)^{(k)}(y')| \le C_k\ \text{for all} \ y'\in [0,\eta^s],\ k\ge 4,$ we could immediately argue as in \cite{v05}, or even apply directly Theorem 1.1 in \cite{lee05},
to prove the following bilinear estimate
\begin{equation}\label{Es<}
\|\ext^s_{(W_1^<)^s}f\, \ext^s_{(W_2^<)^s}g\|_{L^p} \leq C_{r,q} \|f_{(W_1^<)^s}\|_{L^q} \|g_{(W_2^<)^s}\|_{L^q}
\end{equation}
for the scaled surface given as the graph of $\phi^s,$ from which \eqref{E<} follows immediately by scaling back to our original coordinates.

However, this control of the higher order derivatives is here no longer available, only \eqref{Hsderiv}, but we had already seen in Section \ref{bilinprot} how to establish the required  bilinear estimates  even under such weaker assumptions, working with slowly decaying wave packets, and the same reasoning can be applied to our present situation,  and thus we can still verify the required estimates in  \eqref{Es<}.
\end{proof}

Given the estimates \eqref{E>} and \eqref{E<}, we can finally complete the proof of Theorem \ref{mainresult}. We first decompose the interval $[-1,1]$ for any dyadic number $2^{j_1}$ into dyadic subintervals $J^{j_1}_{k}$ of length $d=2^{-j_1}, j_1\ge 1,$ in the usual way and say that two such dyadic subintervals $J ^{j_1}_{k}$  and $J ^{j_1}_{k'}$ are {\it related} and write $J ^{j_1}_{k}\approx J ^{j_1}_{k'}$ if they are not adjacent but have adjacent dyadic parent intervals of length $2^{1-j_1}.$ In a similar way, we decompose the interval $[0,\eta_1]$ into dyadic intervals $I ^{j_2}_{l}$ of length $r=2^{-j_2}$   and define when two such intervals $I ^{j_2}_{l}$ and $I ^{j_2}_{l'}$ are related and write  $I^{j_2}_{l}\approx I^{j_2}_{l'}$ in the same way as before.  Finally, we put  $W^{j_1,j_2}_{k,l}:=J ^{j_1}_{k}\times I ^{j_2}_{l},$ and say that two such rectangles $W^{j_1,j_2}_{k,l}$ and $W^{j_1,j_2}_{k',l'}$  of dimension $d\times r$ are {\it related} and  write $W^{j_1,j_2}_{k,l}\approx W^{j_1,j_2}_{k',l'}$  if  $J ^{j_1}_{k}\approx J ^{j_1}_{k'}$ and $I^{j_2}_{l}\approx I^{j_2}_{l'}.$ This leads to a kind of Whitney decomposition of $\Om\times \Om$ away from its diagonal $D$ into rectangular boxes,  i.e., 
\begin{equation}\label{whitneyrect}
\Om\times \Om\setminus D=\bigcup\limits_{j_1\ge 1,\,  j_2\ge \log_2(1/\la_1)}\bigcup\limits_{W^{j_1,j_2}_{k,l}\approx W^{j_1,j_2}_{k',l'}}W^{j_1,j_2}_{k,l}\times W^{j_1,j_2}_{k',l'}.
\end{equation}
Since $\|\ext f\|^2_{L^r}=\|\ext f \ext g\|_{L^p},$ if we choose $g:=f,$ it will suffice to  estimate $\|\ext f\,  \ext g\|_{L^p}, $ and writing 
$$f^{j_1,j_2}_{k,l}:=f \chi_{W^{j_1,j_2}_{k,l}},\ g^{j_1,j_2}_{k',l'}:=g \chi_{W^{j_1,j_2}_{k',l'}},
$$
by \eqref{whitneyrect} we may decompose 
\begin{equation*}
\ext f\,  \ext g=\sum\limits_{j_1\ge 1,\,  j_2\ge \log_2(1/\la_1)}\sum\limits_{W^{j_1,j_2}_{k,l}\approx W^{j_1,j_2}_{k',l'}} (\ext f ^{j_1,j_2}_{k,l})\,  (\ext g^{j_1,j_2}_{k',l'}),
\end{equation*}
so that 
\begin{equation}\label{edecomp}
\|\ext f \ext g\|_{L^p}\le \sum\limits_{j_1\ge 1,\,  j_2\ge \log_2(1/\la_1)}\Big\| \sum\limits_{W^{j_1,j_2}_{k,l}\approx W^{j_1,j_2}_{k',l'}} (\ext f ^{j_1,j_2}_{k,l})\,  (\ext g^{j_1,j_2}_{k',l'})\Big\|_{L^p}.
\end{equation}
We shall estimate each summand separately. To this end, let us fix $j_1,j_2,$ and  defray the notation by writing $d:=2^{-j_1}, r:=2^{-j_2},$ and $W_{1;k,l}:=W^{j_1,j_2}_{k,l},\,  W_{2;k',l'}:=W^{j_1,j_2}_{k',l'},$ and $f_{k,l}:=f ^{j_1,j_2}_{k,l}, \, g_{k',l'}:=g^{j_1,j_2}_{k',l'}.$ We also shortly write $k\approx k',l\approx l'$ in place of  $W_{1;k,l}\approx W_{2;k',l'}.$ Note that  $|k-k'|\le 2$ and $|l-l'|\le 2,$ if $k\approx k',l\approx l'.$  Then, with the preceding notation, since $d$ and $r$ are given, we may decompose 
\begin{eqnarray*}
\ext f_{k,l}&=&\ext_{W_{1;k,l}^>}f_{k,l}+\ext_{W_{1;k,l}^<}f_{k,l},\\
\ext g_{k',l'}&=&\ext_{W_{2;k',l'}^>}g_{k',l'}+\ext_{W_{2;k',l'}^<}g_{k',l'},\\
\end{eqnarray*}
which leads to a decomposition of $(\ext f ^{j_1,j_2}_{k,l})\,  (\ext g^{j_1,j_2}_{k',l'})$ into four terms, whose contributions we shall compute separately.
\medskip 

We begin by estimating $I:= \|\sum_{k\approx k',l\approx l'} \ext_{W_{1;k,l}^>}f_{k,l} \, \ext_{W_{2;k',l'}^>}g_{k',l'}\|_p.$ To this end, we first apply a standard ''orthogonality'' argument (compare for the proof of Lemma 6.1 in \cite{TVV}), followed by an application of the  Cauchy-Schwarz inequality, to see that 
$$
I^p\lesssim  \sum_{k\approx k',l\approx l'} \|\ext_{W_{1;k,l}^>}f_{k,l} \, \ext_{W_{2;k',l'}^>}g_{k',l'}\|_p^p\le \sum_{k\approx k',l\approx l'} \|\ext_{W_{1;k,l}^>}f_{k,l}\|_r^p \, 
\|\ext_{W_{2;k',l'}^>}g_{k',l'}\|_r^p.
$$
Next, since $|k-k'|\le 2$ and $|l-l'|\le 2,$ applying again Cauchy-Schwarz, we may then essentially estimate $I$  by
$$
I\lesssim (\sum_{k,l} \|\ext_{W_{1;k,l}^>}f_{k,l}\|_r^r)^{1/r} (\sum_{k,l} \|\ext_{W_{1;k,l}^>}g_{k,l}\|_r^r)^{1/r}
$$
(plus at most $8$ terms of similar kind). By \eqref{E>}, this can eventually be estimated by 
\begin{equation}\label{Iest}
I\lesssim \la_1^\ve (cd)^\ve (\sum_{k,l} \|f_{k,l}\|_q^r)^{1/r} (\sum_{k,l} \|g_{k,l}\|_q^r)^{1/r}\le (cd)^\ve\|f\|_q \|g\|_q,
\end{equation}
since $r>q,$ where $\ve >0.$ 
\medskip

We next turn to the term $II:= \|\sum_{k\approx k',l\approx l'} \ext_{W_{1;k,l}^<}f_{k,l} \, \ext_{W_{2;k',l'}^<}g_{k',l'}\|_p.$ Making again use of the afore-mentioned ``othogonality'' argument in combination with the bilinear estimate \eqref{E<}, we find that 
$$
II^p\lesssim (cd)^{\ve p}\sum_{k\approx k',l\approx l'} \|f_{k,l}\|_q^p \, \|g_{k',l'}\|_q^p,
$$
if we assume that $1/q'-1/p>\ve.$ From here on we can argue in a similar way as before to see that also 
\begin{equation}\label{IIest}
II\lesssim (cd)^\ve (\sum_{k,l} \|f_{k,l}\|_q^r)^{1/r} (\sum_{k,l} \|g_{k,l}\|_q^r)^{1/r}\le (cd)^\ve\|f\|_q\|g\|_q.
\end{equation}
\medskip
As for the two ``mixed term'' sums $III $ and $IV$, which can be handled in analogous ways, let is just look at one of them, say 
$III:= \|\sum_{k\approx k',l\approx l'} \ext_{W_{1;k,l}^>}f_{k,l} \, \ext_{W_{2;k',l'}^<}g_{k',l'}\|_p.$ We can  first argue as  for $I$ to see that 
\begin{eqnarray*}
III^p\le  \sum_{k\approx k',l\approx l'} \|\ext_{W_{1;k,l}^>}f_{k,l} \, \ext_{W_{2;k',l'}^<}g_{k',l'}\|_p^p\le \sum_{k\approx k',l\approx l'} \|\ext_{W_{1;k,l}^>}f_{k,l}\|_r^p \, 
\|\ext_{W_{2;k',l'}^<}g_{k',l'}\|_r^p.
\end{eqnarray*}
Here, by \eqref{E>}, we can again estimate
$$
\|\ext_{W_{1;k,l}^>}f_{k,l}\|_r^p\lesssim (\la_1^\ve(cd)^\ve \|f_{k,l}\|_q)^p,
$$
but we do not know a priori that the operator $\ext_{W_{2;k',l'}^<}$ is bounded. 

Therefore we first replace the operator $\ext $ in the preceding estimates by the truncated operator 
$\ext^{\la_0}:=\ext_{A_{\la_0}},$ for any positive dyadic number $0<\la_0<\la_1,$ where $A_{\la_0}:=\{(x,y)\in \Om: h'''(y)>\la_0\}.$ Then, by \eqref{E>}, 
\begin{equation}\label{E>2} 
\|\ext^{\la_0} f\|_{L^r} \leq C_{r,q}(\la_0) \|f\|_{q},\qquad i=1,2,
\end{equation}
and we choose for $C_{r,q,\la_0}$ the smallest possible constant for this estimate. Then we can also estimate 
$\|\ext^{\la_0}_{W_{2;k',l'}^<}g_{k',l'}\|_r\le C_{r,q}(\la_0)\|g_{k',l'}\|_q,$ and if we denote $III_{\la_0}$ the same expression as $III,$ only with $\ext$ replaced by $\ext^{\la_0},$ then we find again by the Cauchy-Schwarz' inequality and a similar reasoning as before that 
\begin{equation}\label{IIIest}
III_{\la_0}\le  C  \la_1^\ve(cd)^\ve C_{r,q}(\la_0)\|f\|_q\|g\|_q.
\end{equation}
The same kind of estimate also holds for the second ''mixed term'' $IV_{\la_0}.$ Combing the latter   estimate and  \eqref{Iest}--\eqref{IIIest}, and choosing $f=g,$ we then find that 
$$
\Big\| \sum\limits_{W^{j_1,j_2}_{k,l}\approx W^{j_1,j_2}_{k',l'}} (\ext f ^{j_1,j_2}_{k,l})\,  (\ext g^{j_1,j_2}_{k',l'})\Big\|_{p}
\le C  (dr)^\ve \big[1+\la_1^\ve C_{r,q}(\la_0)\big]\|f\|_q\|g\|_q.
$$
Summing finally over all dyadic $d=2^{-j_1}$ and $r=2^{-j_2},$ by \eqref{edecomp} we find that 
$$
\|\ext_{\la_0} f\|_r\le C \big[1+\la_1^\ve C_{r,q}(\la_0)\big]^{\frac 12}\|f\|_q.
$$ 
Thus 
$$C_{r,q}(\la_0)\le C \big[1+\la_1^\ve C_{r,q}(\la_0)\big]^{\frac 12},$$
and choosing $\la_1$ sufficiently  small, we see that $C_{r,q}(\la_0)$ is uniformly bounded in $\la_0,$ i.e., there is a constant $C_{r,q}$ such that 
\begin{equation}\label{extuni}
\|\ext_{\la_0} f\|_r\le C \|f\|_q \qquad \text{for all} \ 0<\la_0\le \la_1.
\end{equation}
Since $\ext_{\la_0} f\to \ext f$ uniformly on compact sets as $\la_0\to 0 ,$ this implies that also  $\|\ext f\|_r\le C \|f\|_q,$  so that the proof of Theorem \ref{mainresult} is complete.

\bigskip

\thispagestyle{empty}

\renewcommand{\refname}{References}


\begin{thebibliography}{----------}


  \bibitem [Be16]{be16}  Bejenaru, I.,  Optimal bilinear restriction estimates for general
    hypersurfaces and the role of the shape operator. Int. Math. Res. Not. IMRN  (2017),  no. 23, 7109--7147.

\bibitem [Bo91] {Bo1} Bourgain, J., Besicovitch-type maximal operators and applications to
    Fourier analysis.  Geom. Funct. Anal. 22 (1991), 147--187.
\bibitem [Bo95a] {Bo2}  Bourgain, J.,  Some new estimates on oscillatory integrals. Essays
    in Fourier Analysis in honor of E. M. Stein. Princeton Math. Ser. 42, Princeton
    University Press, Princeton, NJ 1995, 83--112.
\bibitem [Bo95b]  {Bo3}  Bourgain, J.,  Estimates for cone multipliers. Oper. Theory Adv. Appl. 77 (1995), 1--16.
\bibitem [BoG11] {BoG}  Bourgain, J.,  Guth, L., Bounds on oscillatory integral operators based on multilinear estimates. Geom. Funct. Anal., Vol.21 (2011)
    1239--1295.
\bibitem [BMV16] {bmv16}  Buschenhenke, S., M\"uller, D.,  Vargas,  A., A Fourier
    restriction theorem for a two-dimensional surface of finite type. Anal. PDE 10-4 (2017), 817--891.
\bibitem [BMV17] {bmv17}  Buschenhenke, S., M\"uller, D.,  Vargas,  A., A Fourier
    restriction theorem for a perturbed  hyperbolic paraboloid.   Proc. London Math. Soc.   (3) 120 (2020), no. 1, 124--154.
\bibitem [BMV19] {bmv19}  Buschenhenke, S., M\"uller, D.,  Vargas,  A.,
  On  Fourier restriction for finite-type  perturbations of the  hyperbolic paraboloid,
  preprint 2019,  arXiv:1902.05442v2.
\bibitem[ChL17]{chl17} Cho, C.-H.,  Lee, J.,  Improved restriction estimate for hyperbolic
    surfaces in $\R^3$  . J. Funct. Anal.  273  (2017),  no. 3, 917--945.
     \bibitem [Gr81] {Gr} Greenleaf, A., Principal Curvature and Harmonic Analysis. Indiana
    Univ. Math. J. Vol. 30, No. 4 (1981).
\bibitem[Gu16]{Gu16}  Guth, L.  A restriction estimate using polynomial partitioning. J.
    Amer. Math. Soc.  29  (2016),  no. 2, 371--413.
    \bibitem[Gu17]{Gu17}  Guth, L.,   Restriction estimates using polynomial partitioning
    II. Acta Math. Vol. 221, No. 1 (2016), 81--142.
    \bibitem [IKM10] {ikm}  Ikromov, I. A.,  Kempe, M.,  M\"uller, D.,  Estimates for maximal
    functions associated with hypersurfaces in $\R^3$ and related problems in harmonic
    analysis. Acta Math. 204 (2010), 151--271.
\bibitem[IM11]{IM-uniform}  Ikromov, I. A.,   M\"uller, D.,  Uniform estimates for the
    Fourier transform of surface carried measures in  $\R^3$ and an application to Fourier
    restriction. J. Fourier Anal. Appl., 17 (2011), no. 6, 1292--1332.
\bibitem[IM15] {IM} Ikromov, I. A.,   M\"uller, D., Fourier restriction for hypersurfaces
    in three dimensions and Newton polyhedra. Annals of Mathematics Studies, 194.
    Princeton
    University Press, Princeton, NJ, 2016.

\bibitem [K17]{k17}  Kim, J.,  Some remarks on Fourier restriction estimates, preprint
    2017. arXiv:1702.01231
\bibitem [L05] {lee05}  Lee, S.,  Bilinear restriction estimates for surfaces with
    curvatures of different signs, Transactions of the American Mathematical Society, Vol.
    358, No. 8, 3511--2533, 2005.
    \bibitem [LV10] {lv10} Lee, S., Vargas, A.,  Restriction estimates for some surfaces with
    vanishing curvatures. J. Funct. Anal.  258  (2010),  no. 9, 2884--2909.
    \bibitem [MVV96] {MVV1} Moyua, A., Vargas, A.,  Vega, L., Schr\"odinger maximal function
    and restriction
properties of the Fourier transform. Internat. Math. Res. Notices 16 (1996), 793--815.
\bibitem [MVV99] {MVV2} Moyua, A., Vargas, A.,  Vega, L., Restriction theorems and maximal
    operators related to oscillatory integrals in $\R^3$. Duke Math. J., 96 (3), (1999),
    547--574.
\bibitem [SJ 74]{sj74} Sj\"olin, Per,  Fourier multipliers and estimates of the Fourier transform of
   measures carried by smooth curves in $\R^{2},$ Studia Math.  51  (1974), 169--182.
\bibitem [St86] {St1}  Stein, E.M., Oscillatory Integrals in Fourier Analysis. Beijing
    Lectures in Harmonic Analysis. Princeton Univ. Press 1986.

      \bibitem [Sto17a] {Sto17a}  Stovall, B.,  Linear and bilinear restriction to certain rotationally
    symmetric hypersurfaces. Trans. Amer. Math. Soc.  369  (2017),  no. 6, 4093--4117.
     \bibitem [Sto17b] {Sto17b}  Stovall, B.,  Scale invariant Fourier restriction to a hyperbolic
    surface. Anal. PDE 12 (2019), no. 5, 1215--1224.

    \bibitem [Str77] {Str} Strichartz, R. S., Restrictions of Fourier transforms to quadratic
    surfaces and decay of solutions of wave equations. Duke Math. J.  44  (1977), no. 3,
    705--714.

    \bibitem [T01] {T4} Tao, T.,  Endpoint bilinear restriction theorems for the cone, and
    some sharp null-form estimates. Math. Z. 238 (2001),215--268.
    \bibitem [T03] {T2} Tao, T., A Sharp bilinear restriction estimate for paraboloids.
    Geom. Funct. Anal. 13, 1359--1384, 2003.
 \bibitem [TVV98] {TVV} Tao, T., Vargas, A., Vega, L., A bilinear approach to the
    restriction and Kakeya conjectures. J. Amer. Math. Soc. 11 (1998) no. 4 , 967--1000.
    \bibitem [TVI00] {TV1} Tao, T.,  Vargas, A.,  A bilinear approach to cone multipliers I.
    Restriction estimates. Geom. Funct. Anal. 10, 185--215, 2000.
\bibitem [TVII00] {TV2} Tao, T.,  Vargas, A.,  A bilinear approach to cone multipliers II.
    Applications. Geom. Funct. Anal. 10, 216--258, 2000.
    \bibitem [To75] {To}  Tomas,  P. A., A restriction theorem for the Fourier transform.
    Bull. Amer. Math. Soc. 81 (1975), 477--478.


\bibitem [V05]{v05}  Vargas, A.,  Restriction theorems for a surface with negative
    curvature, Math. Z. 249, 97--111 (2005).
    \bibitem [W01] {W2} Wolff, T.,  A Sharp Bilinear Cone Restriction Estimate. Ann. of Math., Second Series, Vol. 153, No. 3, 661--698, 2001.
\end{thebibliography}
\end{document}